\documentclass[reqno]{amsart}
\usepackage{esint}
\usepackage{cite}
\usepackage{amssymb}
\usepackage[usenames, dvipsnames]{color}
\usepackage{marginnote}
\usepackage{todonotes}
\usepackage{hyperref}

\usepackage{verbatim}
\newcommand{\RN}[1]{%
  \textup{\uppercase\expandafter{\romannumeral#1}}%
}

\newcommand{\cprime}{'}
\theoremstyle{plain}
\newtheorem{theorem}{Theorem}[section]
\newtheorem{lemma}[theorem]{Lemma}
\newtheorem{corollary}[theorem]{Corollary}
\newtheorem{proposition}[theorem]{Proposition}

\theoremstyle{definition}
\newtheorem{definition}[theorem]{Definition}

\theoremstyle{remark}
\newtheorem{remark}[theorem]{Remark}

\providecommand{\keywords}[1]{\textbf{Key words.} #1}
\renewcommand{\div}{\text{div}}
\newcommand{\Range}{\text{Range}}
\renewcommand{\epsilon}{\varepsilon}
\newcommand{\R}{\mathbb{R}}
\numberwithin{equation}{section}

\makeatletter
\@namedef{subjclassname@2020}{%
  \textup{2020} Mathematics Subject Classification}
\makeatother

\begin{document}
\title[Gradient estimates for $p$-Laplace type equations]{Gradient estimates for singular $p$-Laplace type equations with measure data}

\author[H. Dong]{Hongjie Dong}
\address[H. Dong]{Division of Applied Mathematics, Brown University, 182 George Street, Providence, RI 02912, USA}

\email{Hongjie\_Dong@brown.edu}

\thanks{H. Dong is partially supported by the Simons Foundation, grant \# 709545.}

\author[H. Zhu]{Hanye Zhu}
\address[H. Zhu]{Division of Applied Mathematics, Brown University, 182 George Street, Providence, RI 02912, USA}

\email{Hanye\_Zhu@brown.edu}



\begin{abstract}
We are concerned with interior and global gradient estimates for solutions to a class of singular quasilinear elliptic equations with measure data, whose prototype is given by the $p$-Laplace equation $-\Delta_p u=\mu$ with $p\in (1,2)$. The cases when $p\in \big(2-\frac 1 n,2\big)$ and $p\in \big(\frac{3n-2}{2n-1},2-\frac{1}{n}\big]$ were
studied in \cite{duzaar2010gradient} and \cite{nguyen2020pointwise}, respectively. In this paper, we improve the results in \cite{nguyen2020pointwise} and address the open case when $p\in \big(1,\frac{3n-2}{2n-1}\big]$. Interior and global modulus of continuity estimates of the gradients of solutions are also established.
\end{abstract}
\maketitle

\section{Introduction}
In this paper, we consider the quasilinear elliptic equation with measure data
\begin{equation}\label{eq:u}
    -\div(A(x,\nabla u))=\mu
\end{equation}
in a domain $\Omega\subset \R^n$, where $n\geq 2$. Here $\mu$ is a locally finite signed Radon measure in $\Omega$, namely, $|\mu|(B_R(x)\cap \Omega)<\infty$ for any ball $B_R(x)\subset \mathbb{R}^n$.  By setting $|\mu|(\R^n\backslash\Omega)$=0, we will always assume that $\mu$ is defined in the whole space $\R^n$. The vector field $A=(A_1,\ldots,A_n):\Omega\times \R^n\rightarrow \R^n$ is assumed to satisfy the following growth, ellipticity, and continuity conditions: there exist constants $\lambda\geq 1$, $s\geq 0$, and $p>1$ such that
\begin{equation}\label{ineq:growth}
    |A(x,\xi)|\leq\lambda(s^2+|\xi|^2)^{(p-1)/2}, \quad |D_\xi A(x,\xi)|\leq\lambda(s^2+|\xi|^2)^{(p-2)/2},
\end{equation}
\begin{equation}\label{ineq:elliptic}
    \left\langle D_\xi A(x,\xi)\eta,\eta \right\rangle\geq\lambda^{-1}(s^2+|\xi|^2)^{(p-2)/2}|\eta|^2,
\end{equation}
and
\begin{equation}\label{ineq:osi}
    |A(x,\xi)-A(x_0,\xi)|\leq\lambda\omega(|x-x_0|)(s^2+|\xi|^2)^{(p-1)/2}
\end{equation}
hold for every $x,x_0\in\Omega$ and every $(\xi,\eta)\in\mathbb{R}^n\times\mathbb{R}^n\backslash\{(0,0)\}$, and $\omega:[0,\infty)\rightarrow[0,1]$ is a
concave non-decreasing function satisfying
$$
\lim_{r\rightarrow0^+}\omega(r)=\omega(0)=0
$$
and the Dini condition
\begin{equation}\label{dini}
    \int_0^1\omega(r)\,\frac{dr}{r}<+\infty.
\end{equation}

A typical model equation is given by the (possibly nondegenerate) $p$-Laplace equation with measure data and $s\geq 0$:
\begin{equation}\label{p-eq}
    -\div\left(a(x)(|\nabla u|^2+s^2)^\frac{p-2}{2}\nabla u\right)=\mu \quad \text{in} \,\, \Omega,
\end{equation}
where $a(\cdot)$ is a Dini continuous function in $\Omega$, satisfying
\begin{equation}\label{bdd}
0<\lambda^{-1}\leq a(x)\leq \lambda
\end{equation}
and
\begin{equation}\label{dini-cts}
|a(x)-a(x_0)|\leq \lambda\omega(|x-x_0|)
\end{equation}
for every $x,x_0\in \Omega$.

By a (weak) solution to Eq. \eqref{eq:u}, we mean a function $u\in W^{1,p}_{\text{loc}}(\Omega)$ such that
the distributional relation
$$\int_{\Omega}\langle A(x,\nabla u),D\varphi\rangle\, dx=\int_{\Omega}\varphi \;d\mu$$
holds whenever $\varphi\in C^\infty_0(\Omega)$ has compact support in $\Omega$. We use $B_R(x)$ to denote the open ball of radius $R$ centered at $x$ and we set
$$
B_R=B_R(0),\quad
\Omega_R(x)=\Omega\cap B_R(x).
$$

The gradient estimates for the super-quadratic case when $p\geq 2$ were well studied in the literature. See \cite{MR2729305,MR2823872,MR2746772,MR3004772}. However, the corresponding results for the singular case when $p\in(1,2)$ are far from complete.

In this paper, we are concerned with only the singular case when $p\in (1,2)$.

For singular quasi-linear equations, the case when $p\in\big(2-\frac{1}{n},2\big)$ was considered in the pioneering work \cite{duzaar2010gradient}, in which the authors proved that under the conditions \eqref{ineq:growth}-\eqref{dini}, if $u\in C^1(\Omega)$ solves (\ref{eq:u}), then it holds that
\begin{align*}
    |\nabla u(x)|\leq C\big[\mathbf{I}_1^R(|\mu|)(x)\big]^\frac{1}{p-1}+C\fint_{B_R(x)}(|\nabla u(y)|+s)\,dy
\end{align*}
for every ball $B_R(x)\subset\Omega$ with $R\in (0,1]$, where $C=C(n,p,\lambda,\omega)$. Here $\fint_E$ stands for the integral average over a measurable set $E$ and
\begin{equation}\label{def:riesz}
    \mathbf{I}_1^R(|\mu|)(x)=\int_0^R\frac{|\mu|(B_t(x))}{t^{n-1}}\,\frac{dt}{t}
\end{equation}
is the truncated first-order Riesz potential.
Later, the case when $p\in \big(\frac{3n-2}{2n-1},2-\frac{1}{n}\big]$ was treated in \cite{nguyen2020pointwise}, in which the authors obtained a pointwise gradient bound involving the Wolff potential under stronger assumptions on $A$ and $\omega$.  Namely, under the conditions \eqref{ineq:growth}-\eqref{ineq:osi} and further assuming that
\begin{equation}
                        \label{eq3.31}
\begin{aligned}
&|D_\xi A(x,\xi)-D_{\xi} A(x,\eta)|\\
&\leq \lambda (s^2+|\xi|^2)^{(p-2)/2}(s^2+|\eta|^2)^{(p-2)/2} (s^2+|\xi|^2+|\eta|^2)^{(2-p-\alpha_0)/2}|\xi-\eta|^{\alpha_0}
\end{aligned}
\end{equation}
and
$$
\int_0^1\omega^\gamma(r)\,\frac{dr}{r}<+\infty
$$
for some $\alpha_0\in(0,2-p)$ and $\gamma\in \Big(\frac{n}{2n-1},\frac{n(p-1)}{n-1}\Big)\subset (0,1)$, if $u\in C^1(\Omega)$ is a solution to (\ref{eq:u}), then it holds that
\begin{align*}
    |\nabla u(x)|\leq C
    [\mathbf{P}_{\gamma}^R(|\mu|)(x)]^\frac{1}{\gamma (p-1)}+C\Big(\fint_{B_R(x)}(|\nabla u(y)|+s)^{\gamma}\,dy\Big)^\frac{1}{\gamma},
\end{align*}
where $C=C(n,p,\lambda,\alpha_0,\omega,\gamma)$ and
$$
\mathbf{P}_{\gamma}^R(|\mu|)(x)=\int_0^R
    \Big(\frac{|\mu|(B_t(x))}{t^{n-1}}\Big)^{\gamma}\,\frac{dt}{t}
$$
is a truncated nonlinear Wolff potential. We recall that in general, the truncated Wolff potential is defined as
\begin{equation}\label{def:wolff}
   \mathbf{W}_{\beta,p}^R(|\mu|)(x)=\int_0^R\left(\frac{|\mu|(B_t(x))}{t^{n-\beta p}}\right)^\frac{1}{p-1}\,\frac{dt}{t}, \quad \beta\in(0,n/p].
\end{equation}

Our first main result is stated as follows.
\begin{theorem}[Interior pointwise gradient estimate]\label{thm:int}
Let $p\in \big(\frac{3n-2}{2n-1},2\big)$ and suppose that $u\in W^{1,p}_{\text{loc}}(\Omega)$ is a solution to \eqref{eq:u}. Then under the assumptions \eqref{ineq:growth}-\eqref{dini}, there exists a constant $C=C(n, p, \lambda,\omega)$ such that the estimate
\begin{equation}\label{ineq:int}
    |\nabla u(x)|\leq C\big[\mathbf{I}_1^R(|\mu|)(x)\big]^\frac{1}{p-1}+C\left(\fint_{B_R(x)}(|\nabla u(y)|+s)^{2-p} \,dy\right)^\frac{1}{2-p}
\end{equation}
holds for any  Lebesgue point $x$ of the vector-valued function $\nabla u$, with $B_R(x)\subset\Omega$ and $R\in (0, 1]$.
\end{theorem}

\begin{remark}
Our pointwise bound in Theorem \ref{thm:int} using Riesz potential $\mathbf{I}_1^R(|\mu|)$ is  an improvement of the bound in \cite[Theorem 1.1]{nguyen2020pointwise} which contains the Wolff potential $\mathbf{P}_{\gamma}^R(|\mu|)$, since
$$
\mathbf{I}_1^R(|\mu|)\leq C\; \mathbf{P}_{\gamma}^{2R}(|\mu|)^\frac{1}{\gamma}\quad \forall\, \gamma<1.
$$
The conditions on $\omega$ and $A$ in Theorem \ref{thm:int} are also weaker. In particular, \eqref{eq3.31} is not assumed.
\end{remark}

For the more singular case when $p\in \big(1,\frac{3n-2}{2n-1}\big]$, which was open, we obtain the following Lipschitz estimate.
\begin{theorem}[Interior Lipschitz estimate]\label{thm:int:lip}
Let $p\in(1,2)$ and suppose that $u\in W^{1,p}_{\text{loc}}(\Omega)$ is a solution to \eqref{eq:u}. Then under the assumptions \eqref{ineq:growth}-\eqref{dini}, there exists a constant $C=C(n, p, \lambda,\omega)$ such that the estimate
\begin{equation}\label{ineq:int:lip}
   \|\nabla u\|_{L^\infty(B_{R/2}(x))}
         \leq C \big\|\mathbf{I}_1^R(|\mu|)\big\|^\frac{1}{p-1}_{L^\infty(B_{R}(x))}+C R^{-\frac{n}{2-p}} \||\nabla u|+s\|_{L^{2-p}(B_{R}(x))}
\end{equation}
holds for any $B_R(x)\subset\Omega$ and $R\in (0, 1]$.
\end{theorem}

We also obtain a modulus of continuity estimate of $\nabla u$ in Theorem \ref{thm4.1}, which directly implies the following sufficient condition for the continuity of $\nabla u$.

\begin{theorem}[Gradient continuity via Riesz potential]\label{thm1.2}
Let $p\in(1,2)$ and $u\in W^{1,p}_{\text{loc}}(\Omega)$ be a solution to \eqref{eq:u}. Assume that \eqref{ineq:growth}-\eqref{dini} are satisfied and the functions
\begin{equation}\label{asp1}
x\;\to \mathbf{I}_1^R(|\mu|)(x) \text{ converge locally uniformly to zero in } \Omega \text{ as } R\;\to\;0.
\end{equation}
Then $\nabla u$ is continuous in $\Omega$.
\end{theorem}

Recall the Lorentz space $L^{n,1}$ is the collection of measurable functions $f$ such that
$$
\int_0^\infty |\{x:\,|f(x)|\ge t\}|^{\frac 1 n}\,{dt}<\infty.
$$
Theorem \ref{thm1.2} has the following corollary.

\begin{corollary}[Gradient continuity via Lorentz spaces]\label{thm1.3}
Let $p\in(1,2)$ and $u\in W^{1,p}_{\text{loc}}(\Omega)$ be a solution to \eqref{eq:u}. Assume that
\eqref{ineq:growth}-\eqref{dini} are satisfied and
\begin{equation}\label{asp2}
  \mu\in L^{n,1} \text{ holds locally in }\Omega.
\end{equation}
Then $\nabla u$ is continuous in $\Omega$.
\end{corollary}
We remark that the Lorentz-space result above was proved in \cite{MR3247381} for a $p$-Laplacian system similar to \eqref{p-eq} when $p\in (1,\infty)$.

A further, actually immediate, corollary of Theorem \ref{thm1.2} concerns measures with certain density properties.
\begin{corollary}[Gradient continuity via density]\label{thm1.4}
Let $p\in(1,2)$ and $u\in W^{1,p}_{\text{loc}}(\Omega)$ be a solution to \eqref{eq:u}. Assume that \eqref{ineq:growth}-\eqref{dini} are satisfied and $\mu$ satisfies
\begin{equation}\label{asp3}
|\mu|(B_{\rho}(x))\leq C \rho^{n-1}h(\rho)
\end{equation}
for every ball $B_{\rho}(x)\subset\subset\Omega$, where $C$ is a positive constant and $h:[0,\infty)\to[0,\infty)$ is a function satisfying the Dini condition
\begin{equation}\label{asp4}
\int_0^R h(r)\;\frac{dr}{r}<\infty \text{ for some }R>0.
\end{equation}
Then $\nabla u$ is continuous in $\Omega$.
\end{corollary}

We remark that Theorem \ref{thm1.2} and Corollaries \ref{thm1.3} and \ref{thm1.4} above are indeed the sub-quadratic --$p\in(1,2)$-- counterparts of \cite[Theorem 1, 3, 4]{MR2729305}. See also \cite[Theorem 1.5 and Corollaries 1.6 and 1.7]{MR3004772}. We refer the reader to \cite{MR2729305} for a discussion of
the borderline nature of the assumptions in these results.

Another interesting consequence of Theorem \ref{thm4.1} is the following gradient H\"older continuity result.

\begin{corollary}[Gradient H\"older continuity via Riesz potential]\label{thm1.8}
Let $p\in(1,2)$ and $u\in W^{1,p}_{\text{loc}}(\Omega)$ be a solution to \eqref{eq:u}. Then under the assumptions \eqref{ineq:growth}-\eqref{dini}, there exists a constant $\alpha\in(0,1)$ depending only on $n$, $p$, and $\lambda$, such that if $\omega(r)\leq C r^{\beta}$ whenever $r>0$ and
$\mathbf{I}_1^{\rho}(|\mu|)(x)\leq C \rho^{\beta}$ whenever $B_{\rho}(x)\subset\subset\Omega$,
for some constants $C>0$ and $\beta\in (0,\alpha)$, then $u\in C^{1,\beta}_{\text{loc}}(\Omega)$.
\end{corollary}

\begin{remark}
We should stress that the constant $\alpha$ in Corollary \ref{thm1.8} is the natural H\"older exponent of the gradients of solutions to corresponding homogeneous equations with $x$-independent nonlinearities (cf. Lemma \ref{lem2.2}). Therefore, our result in Corollary \ref{thm1.8} provides the best possible H\"older exponent for the gradient of the solution. The previous Corollary is an improvement of the gradient H\"older regularity result by Liebermann \cite[Theorem 5.3]{MR1233190}, who proved $u\in C^{1,\beta_1}_{\text{loc}}$ for some $\beta_1=\beta_1(n,p,\lambda,\beta)\in(0,1)$ assuming that $\omega(r)\leq C r^{\beta}$ and $|\mu|(B_{\rho}(x))\leq C \rho^{n-1+\beta}$ for some $\beta\in (0,1)$. It is easily seen that the last condition implies $\mathbf{I}_1^{\rho}(|\mu|)(x)\leq C \rho^{\beta}$ whenever $B_{\rho}(x)\subset\subset\Omega$.
\end{remark}

We also obtain up-to-boundary gradient estimates for the $p$-Laplace equations with measure data in domains with $C^{1,\text{Dini}}$ boundaries.

\begin{definition}\label{def:bdry}
Let $\Omega$ be a domain in $\R^n$. We say that $\Omega$ has $C^{1,\text{Dini}}$ boundary if there exists a constant $R_0\in(0,1]$ and a non-decreasing function $\omega_0:[0,1]\rightarrow[0,1]$  satisfying the Dini condition
$$
\int_0^1\omega_0(r)\,\frac{dr}{r}<+\infty,
$$
such that the following holds: for any $x_0=(x_{01},x^{\prime}_0)\in \partial \Omega$, there exists a $C^{1,\text{Dini}}$ function (i.e., $C^1$ function whose first derivatives are uniformly Dini continuous) $\chi:\mathbb{R}^{n-1}\rightarrow\mathbb{R}$ and a coordinate system depending on $x_0$ such that
\begin{equation*}
\sup_{|x'_1-x'_2|\leq r} |\nabla_{x'}\chi(x'_1)-\nabla_{x'}\chi(x'_2)|\leq \omega_0(r),\quad \forall r\in(0,R_0),
\end{equation*}
and that in the new coordinate system, we have
\begin{equation*}
|\nabla_{x'}\chi(x_0^{\prime})|=0, \quad \Omega_{R_0}(x_0)=\{x\in B_{R_0}(x_0):x_1>\chi(x')\}.
\end{equation*}
\end{definition}

Our global pointwise gradient estimate and Lipschitz estimate are stated as follows.
\begin{theorem}[Boundary pointwise gradient estimate]\label{thm:bdry}
Let $p\in \big(\frac{3n-2}{2n-1},2\big)$ and suppose that $u\in W^{1,p}_0(\Omega)$ is a solution to \eqref{p-eq} with Dirichlet boundary data $u=0$ on $\partial \Omega$. Assuming that \eqref{dini}, \eqref{bdd}, and \eqref{dini-cts} are satisfied and $\Omega$ has a $C^{1,Dini}$ boundary characterized by $R_0$ and $\omega_0$ as in Definition \ref{def:bdry}, then there exists a constant $C=C(n, p, \lambda,\omega,R_0,\omega_0)$ such that estimate
\begin{equation}\label{ineq:bdry1}
    |\nabla u(x)|\leq C[\mathbf{I}^R_1(|\mu|)(x)]^\frac{1}{p-1}+C\left(\fint_{\Omega_R(x)}(|\nabla u(y)|+s)^{2-p} \,dy\right)^\frac{1}{2-p}
\end{equation}
holds for any Lebesgue point $x\in {\Omega}$ of the vector-valued function $\nabla u$ and $R\in (0,1]$. Moreover, if $u\in C^1(\Bar{\Omega})$, the estimate \eqref{ineq:bdry1} holds for any $x\in \Bar{\Omega}$.
\end{theorem}

\begin{theorem}[Boundary Lipschitz estimate]\label{thm:bdry:lip}
Let $p\in (1,2)$ and suppose that $u\in W^{1,p}_0(\Omega)$ is a solution to \eqref{p-eq} with Dirichlet boundary data $u=0$ on $\partial \Omega$. Assuming that \eqref{dini}, \eqref{bdd}, and \eqref{dini-cts} are satisfied and $\Omega$ has a $C^{1,Dini}$ boundary characterized by $R_0$ and $\omega_0$ as in Definition \ref{def:bdry}, then there exists a constant $C=C(n, p, \lambda,\omega,R_0,\omega_0)$ such that the estimate
\begin{equation}\label{ineq:bdry:lip}
\|\nabla u\|_{L^\infty(\Omega_{R/2}(x))}
         \leq C\big\|\mathbf{I}_1^R(|\mu|)\big\|^\frac{1}{p-1}_{L^\infty(\Omega_{R}(x))}+ CR^{-\frac{n}{2-p}} \||\nabla u|+s\|_{L^{2-p}(\Omega_{R}(x))}
  \end{equation}
holds for any $x\in \Bar{\Omega}$ and $R\in (0,1]$.
\end{theorem}

As a corollary, we also obtain the global Lipschitz estimate when $\Omega$ is bounded.
\begin{corollary}\label{thm:glo:lip}
Let $\Omega\subset \R^n$ be a bounded domain. Under the assumptions of Theorem \ref{thm:bdry:lip}, there exists a constant $C=C(n, p, \lambda,\omega,R_0,\omega_0,\text{diam}(\Omega))$ such that
$$ \|\nabla u\|_{L^\infty(\Omega)}
         \leq C\big\|\mathbf{I}^{1}_1(|\mu|)\big\|^\frac{1}{p-1}_{L^\infty(\Omega)}+Cs.
         $$
\end{corollary}

A global modulus of continuity estimate is also established in Theorem \ref{thm:bdry2} under the same conditions, which implies that corresponding up-to-boundary gradient continuity results as in Theorem \ref{thm1.2} and Corollaries \ref{thm1.3}-\ref{thm1.8} also hold for the $p$-Laplace equations with measure data.

Let us give a brief description of the proofs. We first apply an iteration argument to get an $L^{\gamma_0}$-mean oscillation estimate of the gradients of solutions to the homogeneous equation with
$x$-independent nonlinearities
$$
-\div (A_0(\nabla v))=0
$$
in Section \ref{sec2}, where $\gamma_0\in (0,1)$.
Our proofs of the interior gradient estimates are then based on a comparison estimate between the original solution $u$ of \eqref{eq:u} and the solution to the homogeneous equation $ -\div(A(x,\nabla w))=0$ in a ball $B_R$ with the boundary condition $u=w$ on $\partial B_R$. The outcome is the inequality
\begin{equation}
            \label{eq12.06}
\begin{aligned}
         &\left(\fint_{B_{R}}|\nabla u-\nabla w|^{\gamma_0}\,dx \right)^{1/\gamma_0}\\
         &\leq C\left[\frac{|\mu|(B_{R})}{R^{n-1}} \right]^\frac{1}{p-1}+C\frac{|\mu|(B_{R})}{R^{n-1}}\fint_{B_{R}}(|\nabla u|+s)^{2-p}\,dx,
\end{aligned}
\end{equation}
which holds for some constant $\gamma_0\in(0,1)$. The details can be found in Lemma \ref{lem:u-w}. For the case when $p\in \big(\frac{3n-2}{2n-1},2\big)$, we can choose $\gamma_0=2-p$, which is the same integral exponent on the right hand side. We then borrow an idea in \cite{dong2017c1} by estimating the $L^{\gamma_0}$-mean oscillation to adapt the iteration scheme used, for instance, in \cite{duzaar2010gradient}. However, for the case when $p\in \big(1,\frac{3n-2}{2n-1}\big]$, we are only able to prove the comparison estimate \eqref{eq12.06} for some $\gamma_0<2-p$ and that is the reason why we only obtain Lipschitz estimates instead of pointwise gradient estimates in this case.

For the gradient estimates up to the boundary, we use the technique of flattening the boundary and generalize the interior oscillation estimates to half balls. We adapt an idea in \cite{choi2019gradient} to establish the global $L^{\gamma_0}$-mean oscillation estimates by a delicate combination of the interior estimates and the estimates near a flat boundary. To this end, we also apply an odd extension argument to derive an  $L^{\gamma_0}$-mean oscillation estimate on half balls for homogeneous equations with  $x$-independent nonlinearities. This argument only works for equations in diagonalized form, such as the $p$-Laplace equation, so the global estimates for general equations remain open. As a partial result in this direction, we refer the reader to \cite{nguyen2020pointwise} for a weighted pointwise boundary estimate under the condition that $\partial \Omega$ is sufficiently flat in the sense of Reifenberg. We also refer the reader to \cite{MR969499,breit2019global} for boundary regularity results for quasi-linear equations with sufficiently regular right-hand side.

The rest of the paper is organized as follows. In the next section, we derive an $L^{\gamma_0}$-mean oscillation estimate of solutions to the homogeneous equation with $x$-independent nonlinearities. In Section \ref{sec3}, we give the proof of Theorem \ref{thm:int}. Section \ref{sec4} is devoted to the Lipschitz estimate and the interior modulus of continuity estimate of the gradient of solutions as well as its corollaries. Finally, in Section \ref{sec5} we consider the corresponding boundary estimates.

\section{An oscillation estimate}\label{sec2}

This section is devoted to the proof of the following interior oscillation estimate for solutions to the homogeneous equation
\begin{equation}
                        \label{eq1.01}
-\div (A_0(\nabla v))=0\quad \text{in}\,\, \Omega,
\end{equation}
where $A_0=A_0(\xi)$ is a vector field independent of $x$ satisfying conditions (\ref{ineq:growth}) and (\ref{ineq:elliptic}) for some $s\geq0$, $\lambda\geq1$, and $p>1$. In this section, we denote the integral average over $B_R(x)$ by $(\cdot)_{B_R(x)}$.

\begin{theorem}\label{thm:osc}
Let $v\in W_{\text{loc}}^{1,p}(\Omega)$ be a solution to \eqref{eq1.01}
and $\gamma_0\in (0,1)$. Then there exist constants $\alpha\in(0,1)$ depending on $n$, $p$, and $\lambda$, and $C>1$ depending on $n$, $p$, $\lambda$, and $\gamma_0$, such that
for every $B_R(x_0)\subset\Omega$ and $\rho\in (0,R)$, we have
\begin{align}\label{ineq:osc:c}
     \inf_{\mathbf{q}\in \mathbb{R}^n}\left(\fint_{B_{\rho}(x_0)}|\nabla v-\mathbf{q}|^{\gamma_0}\right)^{1/\gamma_0}\leq& C\left(\frac{\rho}{R}\right)^{\alpha}\inf_{\mathbf{q}\in \mathbb{R}^n}\left(\fint_{B_{R}(x_0)}|\nabla v-\mathbf{q}|^{\gamma_0}\right)^{1/\gamma_0}.
\end{align}
\end{theorem}

To prove the above theorem, we first recall a classical oscillation estimate. Estimates of this type, with different exponents involved, were developed in \cite{lieberman1991natural, dibenedetto1993higher, duzaar2010gradient}.
\begin{lemma}
                \label{lem2.2}
Let $v\in W_{\text{loc}}^{1,p}(\Omega)$ be a solution to \eqref{eq1.01}. There exist constants $C>1$ and $\alpha\in(0,1)$ depending only on $n$, $p$, and $\lambda$, such that $v\in C^{1,\alpha}_{\text{loc}}(\Omega)$ and for every $B_R(x_0)\subset\Omega$ and $r\in (0,R)$, we have
\begin{align*}
    \fint_{B_r(x_0)}|\nabla v-(\nabla v)_{B_r(x_0)}|^p \,dx\leq C\left(\frac{r}{R}\right)^{\alpha p} \fint_{B_R(x_0)}|\nabla v-(\nabla v)_{B_R(x_0)}|^p \,dx.
\end{align*}
\end{lemma}

The lemma above directly implies the following corollary.
\begin{corollary}\label{cor}
Under the conditions of Lemma \ref{lem2.2}, there exist constants $C>1$ and $\alpha\in(0,1)$ depending only on $n$, $p$, and $\lambda$, such that for every $B_R(x_0)\subset\Omega$, we have
\begin{equation}\label{eq:holder1}
    R^\alpha[\nabla v]_{C^\alpha(B_{R/2}(x_0))}\leq C\left(\fint_{B_R(x_0)}|\nabla v-(\nabla v)_{B_R(x_0)}|^p\right)^{1/p},
\end{equation}
and for any $r\in [R/2,R)$,
\begin{equation}\label{eq:holder2}
 [\nabla v]_{C^\alpha(B_{r}(x_0))}\leq C\frac{R^{n/p+1-\alpha}}{(R-r)^{n/p+1}}\left(\fint_{B_R(x_0)}|\nabla v-(\nabla v)_{B_R(x_0)}|^p\right)^{1/p}.
\end{equation}
\end{corollary}
\begin{proof}
Without loss of generality, we assume $x_0=0$. For any $x\in B_{R/2}$ and $r\leq R/2$, by Lemma \ref{lem2.2} we have
\begin{align*}
    \fint_{B_r(x)}|\nabla v-(\nabla v)_{B_r(x)}|^p&\leq C\left(\frac{r}{R}\right)^{\alpha p}\fint_{B_{R/2}(x)}|\nabla v-(\nabla v)_{B_{R/2}(x)}|^p\\
    &\leq C\left(\frac{r}{R}\right)^{\alpha p}\fint_{B_{R}}|\nabla v-(\nabla v)_{B_{R}}|^p.
\end{align*}
By Campanato's characterization of H\"older continuous functions, we obtain \eqref{eq:holder1}.

Now for any $r>R/2$ and $z\in B_r$, using (\ref{eq:holder1}) and the triangle inequality, we have
\begin{align}
                    \label{eq6.38}
    [\nabla v]_{C^\alpha(B_{(R-r)/2}(z))}&\leq C(R-r)^{-\alpha}\left(\fint_{B_{R-r}(z)}|\nabla v-(\nabla v)_{B_{R-r}(z)}|^p\right)^{1/p} \notag\\
    &\leq C\frac{R^{n/p}}{(R-r)^{n/p+\alpha}}\left(\fint_{B_R}|\nabla v-(\nabla v)_{B_R}|^p \right)^{1/p}.
\end{align}
Thus for any $x,y\in B_r$, let $N=\min\big\{m\in \mathbb{Z}:m>\frac{2|x-y|}{R-r}\big\}$. We can divide the line segment connecting $x$ and $y$ into $N$ equal segments using $x_1,...,x_{N-1}$, $x_0=x$, and  $x_N=y$, such that
$$
|x_k-x_{k+1}|=\frac{|x-y|}{N}<\frac{R-r}{2}.
$$
Then by the triangle inequality and \eqref{eq6.38}, we have
\begin{align*}
    &|\nabla v(x)-\nabla v(y)|\leq\sum_{k=0}^{N-1} |\nabla v(x_k)-\nabla v(x_{k+1})|\\
    &\leq C\sum_{k=0}^{N-1}\frac{R^{n/p}}{(R-r)^{n/p+\alpha}}\left(\int_{B_R}|\nabla v-(\nabla v)_{B_R}|^p\right)^{1/p}\left|\frac{x-y}{N}
    \right|^\alpha\\
    &\leq CN^{1-\alpha}\frac{R^{n/p}}{(R-r)^{n/p+\alpha}}\left(\int_{B_R}|\nabla v-(\nabla v)_{B_R}|^p\right)^{1/p}|x-y|^\alpha\\
    &\leq C\left(\frac{R}{R-r}\right)^{1-\alpha}\frac{R^{n/p}}{(R-r)^{n/p+\alpha}}\left(\int_{B_R}|\nabla v-(\nabla v)_{B_R}|^p\right)^{1/p}|x-y|^\alpha,
\end{align*}
which directly implies (\ref{eq:holder2}). The corollary is proved.
\end{proof}

Now we are ready to give the proof of Theorem \ref{thm:osc}.
\begin{proof}[Proof of Theorem \ref{thm:osc}]
As before, without loss of generality, we assume $x_0=0$. Clearly for any $B_{\rho}=B_{\rho}(0)\subset \Omega$, there exists $\mathbf{q}_{\rho}=(q_{\rho}^{(1)},...,q_{\rho}^{(n)})\in\mathbb{R}^n$ such that
\begin{align*}
    \left(\fint_{B_{\rho}}|\nabla v-\mathbf{q}_{\rho}|^{\gamma_0}\right)^{1/\gamma_0}
    =\inf_{\mathbf{q}\in\mathbb{R}^n}\left(\fint_{B_{\rho}}|\nabla v-\mathbf{q}|^{\gamma_0}\right)^{1/\gamma_0}.
\end{align*}
Also, it is easily  seen that
\begin{align}\label{range:q}
    q^{(i)}_{\rho}\in \Range(D_iv\mid_{B_\rho}).
\end{align}

We claim that there exists a constant $C$ depending only on $n$, $p$, $\lambda$, and $\gamma_0$, such that
\begin{equation}\label{eq:norm}
    \|\nabla v-\mathbf{q}_{\rho/2}\|_{L^\infty(B_{\rho/2})}\leq C\left(\fint_{B_\rho}|\nabla v-\mathbf{q}_\rho|^{\gamma_0}\right)^{1/\gamma_0}
\end{equation}
holds for any $B_{\rho}(x_0)\subset \Omega$.

We prove the claim by using Corollary \ref{cor} and iteration.

For any $R/2<r<R\leq\text{dist}(x_0,\partial\Omega)$, using (\ref{eq:holder2}) and the triangle inequality, we get
\begin{align*}
    &\|\nabla v-\mathbf{q}_{r}\|_{L^\infty(B_r)}\leq r^\alpha[\nabla v]_{C^\alpha(B_r)}\\
    &\leq C\frac{R^{n/p+1}}{(R-r)^{n/p+1}}\left(\fint_{B_R}|\nabla v-(\nabla v)_{B_R}|^p\right)^{1/p}\\
    &\leq C\frac{R^{n/p+1}}{(R-r)^{n/p+1}}\left(\fint_{B_R}|\nabla v-\mathbf{q}_R|^p\right)^{1/p}\\
    &\leq C\frac{R^{n/p+1}}{(R-r)^{n/p+1}}\|\nabla v-\mathbf{q}_{R}\|^\frac{p-\gamma_0}{p}_{L^\infty(B_{R})}\left(\fint_{B_R}|\nabla v-\mathbf{q}_R|^{\gamma_0}\right)^{1/p}\\
    &\leq \epsilon\|\nabla v-\mathbf{q}_R\|_{L^\infty(B_R)}+C_\epsilon\left(\frac{R^{n/p+1}}{(R-r)^{n/p+1}}\right)^\frac{p}{\gamma_0}\left(\fint_{B_R}|\nabla v-\mathbf{q}_R|^{\gamma_0}\right)^{1/\gamma_0},
\end{align*}
where we used Young's inequality with exponents $p/\gamma_0$ and $p/(p-\gamma_0)$ in the last line.

Now taking $r_k=(1-2^{-k})\rho$, $r=r_k$, and $R=r_{k+1}$, we have
\begin{align*}
    &\|\nabla v-\mathbf{q}_{r_k}\|_{L^\infty(B_{r_k})}\\
    &\leq \epsilon\|\nabla v-\mathbf{q}_{r_{k+1}}\|_{L^\infty(B_{r_{k+1}})}+C_\epsilon2^{k{\beta}}\left(\fint_{B_{r_{k+1}}}|\nabla v-\mathbf{q}_{r_{k+1}}|^{\gamma_0}\right)^{1/\gamma_0}\\
     &\leq \epsilon\|\nabla v-\mathbf{q}_{r_{k+1}}\|_{L^\infty(B_{r_{k+1}})}+C_\epsilon2^{k{\beta}+n/\gamma_0}\left(\fint_{B_{\rho}}|\nabla v-\mathbf{q}_{\rho}|^{\gamma_0}\right)^{1/\gamma_0},
\end{align*}
where ${\beta}=(n+p)/\gamma_0$.
Taking $\epsilon=3^{-{\beta}}$, multiplying both sides by $\epsilon^k$ and summing in $k$, we get
\begin{align*}
    \sum_{k=1}^{\infty}\epsilon^k\|\nabla v-\mathbf{q}_{r_k}\|_{L^\infty(B_{r_k})}\leq&\sum_{k=1}^\infty \epsilon^{k+1}\|\nabla v-\mathbf{q}_{r_{k+1}}\|_{L^\infty(B_{r_{k+1}})}\\&+C\left(\fint_{B_{\rho}}|\nabla v-\mathbf{q}_{\rho}|^{\gamma_0}\right)^{1/\gamma_0},
\end{align*}
where the summations are finite and $C=C(n,p,\lambda,\gamma_0)$.
By subtracting
$$
 \sum_{k=2}^{\infty}\epsilon^k\|\nabla v-\mathbf{q}_{r_k}\|_{L^\infty(B_{r_k})}
$$
from both sides of the above inequality, we obtain \eqref{eq:norm}. The claim is proved.

Now we are ready to prove \eqref{ineq:osc:c}. If $r\leq R/4$, by \eqref{eq:holder1}, \eqref{range:q}, and \eqref{eq:norm} we get
\begin{align*}
    &\left(\fint_{B_r}|\nabla v-\mathbf{q}_r|^{\gamma_0}\right)^{1/\gamma_0}
    \leq  Cr^\alpha[\nabla v]_{C^\alpha(B_\frac{R}{4})}\\
    &\leq C\left(\frac{r}{R}\right)^\alpha \left(\fint_{B_{\frac{R}{2}}}|\nabla v-(\nabla v)_{B_\frac{R}{2}}|^p\right)^{1/p}\leq  C\left(\frac{r}{R}\right)^\alpha \left(\fint_{B_{\frac{R}{2}}}|\nabla v-\mathbf{q}_\frac{R}{2}|^p\right)^{1/p}\\
    &\leq C\left(\frac{r}{R}\right)^\alpha\|\nabla v-\mathbf{q}_\frac{R}{2}\|_{L^\infty(B_\frac{R}{2})}
    \leq C\left(\frac{r}{R}\right)^\alpha \left(\fint_{B_{R}}|\nabla v-\mathbf{q}_R|^{\gamma_0}\right)^{1/\gamma_0}.
\end{align*}
If $r>R/4$, we have
\begin{align*}
     \left(\fint_{B_r}|\nabla v-\mathbf{q}_r|^{\gamma_0}\right)^{1/\gamma_0}
    &\leq  \left(\fint_{B_r}|\nabla v-\mathbf{q}_R|^{\gamma_0}\right)^{1/\gamma_0}\\
    &\leq C \left(\frac{r}{R}\right)^\alpha \left(\fint_{B_{R}}|\nabla v-\mathbf{q}_R|^{\gamma_0}\right)^{1/\gamma_0}.
\end{align*}
The theorem is proved.
\end{proof}

\section{Interior pointwise gradient estimates}\label{sec3}
In order to prove the interior pointwise gradient estimates, we follow the outline of arguments given in \cite{nguyen2020pointwise} while replacing their oscillation estimates with our new oscillation estimate in Section \ref{sec2}. We also borrow an idea in \cite{dong2017c1} by estimating the $L^{\gamma_0}$-mean oscillations of solutions, where $\gamma_0\in (0,1)$.

Let $u\in W^{1,p}_{\text{loc}}(\Omega)$ be a solution to \eqref{eq:u} and $B_{2r}(x_0)\subset\subset \Omega$. We consider the unique solution $w\in u+W_0^{1,p}(B_{2r}(x_0))$ to the equation
\begin{equation}\label{eqa:w}
    \left\{
\begin{aligned}
     -\div(A(x,\nabla w)) =&  0 \quad \text{in} \quad B_{2r}(x_0),  &\\
     w =&  u \quad \text{on}\quad \partial B_{2r}(x_0). &\\
\end{aligned}
\right.
\end{equation}

We first recall an interior reverse H\"older inequality that can be found in \cite[Theorem 6.7]{MR1962933}.
\begin{lemma}[\cite{nguyen2020pointwise}, Lemma 3.1]
Let $w$ be a solution to \eqref{eqa:w}. There exists a constant ${\theta}_1>p$ depending only on $n$, $p$, and $\lambda$ such that for any $t>0$, the estimate
\begin{equation}\label{reverseholder}
    \left (\fint_{B_{\rho/2}(y)}(|\nabla w|+s)^{{\theta}_1}\,dx\right)^\frac{1}{{\theta}_1}\leq C\left(\fint_{B_{\rho}(y)}(|\nabla w|+s)^t\,dx\right)^\frac{1}{t}
\end{equation}
holds for all $B_\rho(y)\subset B_{2r}(x_0)$, where $C=C(n,p,\lambda,t)>0$.
\end{lemma}

We also have the following comparison result, which generalizes and refines similar results in \cite{duzaar2010gradient, nguyen2019good, nguyen2020existence}.

\begin{lemma}\label{lem:u-w}
Let $w$ be a solution to \eqref{eqa:w} and assume that $p\in (1,2)$. Then for any  $\gamma_0\in(0,2-p]$ when $p\in \big(\frac{3n-2}{2n-1},2\big)$ or $\gamma_0\in \big(0,\frac{(p-1)n}{n-1}\big)$ when $p\in \big(1, \frac{3n-2}{2n-1}\big]$, it holds that
    \begin{align*}
         &\left(\fint_{B_{2r}(x_0)}|\nabla u-\nabla w|^{\gamma_0}\,dx \right)^{1/\gamma_0}\\
         &\leq C\left[\frac{|\mu|(B_{2r}(x_0))}{r^{n-1}} \right]^\frac{1}{p-1}+C\frac{|\mu|(B_{2r}(x_0))}{r^{n-1}}\fint_{B_{2r}(x_0)}(|\nabla u|+s)^{2-p}\,dx,
    \end{align*}
where $C$ is a constant depending only on $n$, $p$, $\lambda$, and $\gamma_0$.
\end{lemma}
\begin{proof}
The case when $p\in \big(1, \frac{3n-2}{2n-1}\big]$ and $s=0$ was proved in \cite[Lemma 2.1]{nguyen2020existence} and their proof also works for the case when $p\in \big(1, \frac{3n-2}{2n-1}\big]$ and $s>0$. Therefore, we only focus on the case when $p\in \big(\frac{3n-2}{2n-1},2\big)$ and $s\geq 0$. By scaling invariance (see \cite[Remark 4.1]{duzaar2010gradient} for example), we may assume that $B_{2r}(x_0)=B_2$ and $|\mu|(B_{2})=1$. For $k>0$, using
$$
\varphi_1=T_{2k}(u-w):=\max\big\{\min\{u-w,2k\},-2k\big\}
$$ as a test function in \eqref{eq:u} and \eqref{eqa:w} and recalling \eqref{ineq:elliptic}, we have
\begin{equation}\label{test1}
   \int_{B_2\cap \{|u-w|<2k\}} g^s(u,w)\leq Ck \quad \text{with}\quad g^s(u,w)=\frac{|\nabla(u-w)|^2}{(|\nabla w|+|\nabla u|+s)^{2-p}}.
\end{equation}
By the triangle inequality, we have
\begin{align*}
    |\nabla(u-w)|&=g^s(u,w)^\frac{1}{2}(|\nabla w|+|\nabla u|+s)^\frac{2-p}{2}\\
    &\leq g^s(u,w)^\frac{1}{2}(|\nabla(u -w)|+2|\nabla u|+s)^\frac{2-p}{2}\\
    &\leq Cg^s(u,w)^\frac{1}{2}|\nabla(u- w)|^\frac{2-p}{2}+Cg^s(u,w)^\frac{1}{2}(|\nabla u|+s)^\frac{2-p}{2}.
\end{align*}
Using Young's inequality with exponents $\frac{2}{p}$ and $\frac{2}{2-p}$, we obtain
\begin{equation}\label{ineq:g1}
    |\nabla(u-w)|\leq Cg^s(u,w)^\frac{1}{p}+Cg^s(u,w)^\frac{1}{2} (|\nabla u|+s)^\frac{2-p}{2}.
\end{equation}
Now we set
$$
E_k=B_2\cap \{k<|u-w|<2k\}\quad \text{and}\quad
F_k=B_2\cap \{|u-w|>k\}.
$$
Using the Sobolev inequality, H\"older's inequality, and \eqref{ineq:g1}, we obtain
\begin{equation}\label{ineq:sob1}
\begin{aligned}
    &k|\{x:|u-w|>2k\}\cap B_2|^{\frac{n-1}{n}}\leq C \Big(\int_{B_2}|T_{2k}(u-w)-T_{k}(u-w)|^\frac{n}{n-1}\Big)^\frac{n-1}{n}\\
    &\leq C \int_{E_k}|\nabla (u-w)|\\
    &\leq C \int_{E_k}\Big( g^s(u,w)^\frac{1}{p}+g^s(u,w)^\frac{1}{2}(|\nabla u|+s)^\frac{2-p}{2}\Big)\\
    &\leq C|E_k|^\frac{p-1}{p}\Big(\int_{E_k}g^s(u,w)\Big)^\frac{1}{p}
    +C\Big(\int_{E_k}g^s(u,w)\Big)^\frac{1}{2}\Big(\int_{E_k}(|\nabla u|+s)^{2-p}\Big)^\frac{1}{2}.
\end{aligned}
\end{equation}
From \eqref{test1} and \eqref{ineq:sob1}, we get
$$
k^\frac{1}{2}|F_{2k}|^\frac{n-1}{n}\leq Ck^{-\frac{1}{2}+\frac{1}{p}}|F_k|^\frac{p-1}{p}+CQ_1^\frac{2-p}{2},
$$
where $Q_1:=\||\nabla u|+s\|_{L^{2-p}(B_2)}$.
Therefore, by taking the sup in $k\in (0,\infty)$, we obtain
$$
\|u-w\|_{L^{\frac{n}{2(n-1)},\infty}(B_2)}^{1/2}\leq C \|u-w\|_{L^{\frac{2-p}{2(p-1)},\infty}(B_2)}^{1/p-1/2}+C Q_1^\frac{2-p}{2}.
$$
Since $\frac{3n-2}{2n-1}<p<2$, we have
$$\frac{n}{2(n-1)}>\frac{2-p}{2(p-1)},$$
which implies

$$
\|u-w\|_{L^{\frac{n}{2(n-1)},\infty}(B_2)}^{1/2}\leq C \|u-w\|_{L^{\frac{n}{2(n-1)},\infty}(B_2)}^{1/p-1/2}+C Q_1^\frac{2-p}{2}.
$$
Thus, by Young's inequality, we obtain
\begin{equation}\label{ineq:weak}
    \|u-w\|_{L^{\frac{n}{2(n-1)},\infty}(B_2)}\leq C+C Q_1^\frac{2-p}{2}.
\end{equation}
Let $k,\;l> 0$ and $q=\frac{n}{2(n-1)}$. By the Chebyshev inequality and \eqref{test1}, we have
\begin{align*}
    &|\{x:g^s(u,w)>l\}\cap B_2|\\
    &\leq  |\{x:|u-w|>k\}\cap B_2|+|\{x:|u-w|\leq k,\;g^s(u,w)>l\}\cap B_2|\\
    &\leq  Ck^{-q}\|u-w\|^q_{L^{q,\infty}(B_2)}+\frac{1}{l}\int_{B_2\cap\{x:|u-w|\leq k\}}g^s(u,w)\,dx\\
    &\leq  Ck^{-q}\|u-w\|^q_{L^{q,\infty}(B_2)}+\frac{Ck}{l}.
\end{align*}
By choosing
$$
k=\Big[l\|u-w\|^q_{L^{q,\infty}(B_2)}\Big]^{\frac{1}{1+q}},
$$
we get
$$
l^\frac{q}{1+q} |\{x:g^s(u,w)>l\}\cap B_2|\leq C \|u-w\|^\frac{q}{1+q}_{L^{q,\infty}(B_2)},
$$
Therefore, by taking the sup in $l\in (0,\infty)$, we obtain
\begin{equation}\label{g1:weak}
\|g^s(u,w)\|_{L^{\frac{q}{1+q},\infty}(B_2)}\leq C\|u-w\|_{L^{q,\infty}(B_2)}.
\end{equation}
Let $\gamma_0\in (0,2-p]$. By \eqref{ineq:g1} and H\"older's inequality with exponents $2/p$ and $2/(2-p)$, we get
\begin{equation}\label{ineq:gamma0}
\begin{aligned}
&\int_{B_2}|\nabla(u-w)|^{\gamma_0}\leq C\int_{B_2}\Big(g^s(u,w)^\frac{\gamma_0}{p}+g^s(u,w)^\frac{\gamma_0}{2} (|\nabla u|+s)^\frac{\gamma_0(2-p)}{2}\Big)\\
&\leq C\Big(\int_{B_2}g^s(u,w)^\frac{\gamma_0}{p}\Big)+C\Big(\int_{B_2}g^s(u,w)^{\frac{\gamma_0}{p}}\Big)^\frac{p}{2}\Big(\int_{B_2}(|\nabla u|+s)^{\gamma_0}\Big)^\frac{2-p}{2}\\
&\leq C\|g^s(u,w)\|^{\gamma_0/p}_{L^{\frac{q}{1+q},\infty}(B_2)}
+C\|g^s(u,w)\|^{\gamma_0/2}_{L^{\frac{q}{1+q},\infty}(B_2)} Q_1^{\frac{\gamma_0(2-p)}{2}}.
\end{aligned}
\end{equation}
In the last inequality, we used the fact that
$$
\frac{\gamma_0}{p}<\frac{q}{1+q},\quad \gamma_0\leq 2-p.
$$
Combining \eqref{ineq:weak}, \eqref{g1:weak}, and \eqref{ineq:gamma0}, we have
$$
\int_{B_2}|\nabla(u-w)|^{\gamma_0}\leq C+C Q_1^{\gamma_0(2-p)},
$$
which implies the desired result.
\end{proof}
We now let $v\in w+W_0^{1,p}(B_r(x_0))$ be the unique solution to
\begin{equation}\label{eqa:v}
    \left\{
\begin{aligned}
     -\div(A(x_0,\nabla v)) =&  0 \quad \text{in} \quad B_{r}(x_0), &\\
     v =&  w \quad \text{on}  \quad \partial B_{r}(x_0).  &\\
\end{aligned}
\right.
\end{equation}

By testing \eqref{eqa:w} and \eqref{eqa:v} with $v-w$, we obtain an estimate for the difference $\nabla v-\nabla w$:
\begin{equation}\label{comp}
    \fint_{B_r(x_0)}|\nabla v-\nabla w|^p\,dx\leq C\omega(r)^p\fint_{B_{r}(x_0)}(|\nabla w|+s)^p\,dx.
\end{equation}
Detailed proof
of this result can be found in \cite[Eq. (4.35)]{duzaar2010gradient}.
Thus by (\ref{reverseholder}) and H\"older's inequality, we get
\begin{equation}\label{eq:v-w}
        \fint_{B_r(x_0)}|\nabla v-\nabla w|^{\gamma_0}\,dx\leq C\omega(r)^{\gamma_0}\fint_{B_{2r}(x_0)}(|\nabla w|+s)^{\gamma_0}\,dx.
\end{equation}

For a ball $B_\rho(x)\subset\subset\Omega$ and a function $f\in  W_{\text{loc}}^{1,p}(\Omega)$, there exists $\mathbf{q}_{x,\rho}(f)\in \R^n$ such that
\begin{equation*}
    \left(\fint_{B_\rho(x)}|\nabla f-\mathbf{q}_{x,\rho}(f)|^{\gamma_0}\right)^{1/\gamma_0}=\inf_{\mathbf{q}\in \R^n}\left(\fint_{B_\rho(x)}|\nabla f-\mathbf{q}|^{\gamma_0}\right)^{1/\gamma_0}.
\end{equation*}
We denote $\mathbf{q}_{x,\rho}=\mathbf{q}_{x,\rho}(u)$ and
\begin{equation*}
    \phi(x,\rho)=\inf_{\mathbf{q}\in \R^n}\left(\fint_{B_\rho(x)}|\nabla u-\mathbf{q}|^{\gamma_0}\right)^{1/\gamma_0}.
\end{equation*}
Since
$$
|\mathbf{q}_{x,\rho}-\nabla u(x)|^{\gamma_0}\leq |\mathbf{q}_{x,\rho}-\nabla u(z)|^{\gamma_0}+|\nabla u(z)-\nabla u(x)|^{\gamma_0},
$$
by taking the average over $z\in B_{ \rho}(x)$ and then taking the $\gamma_0$-th root, we obtain
\begin{align*}
    &|\mathbf{q}_{x,\rho}-\nabla u(x)|\leq C\phi(x, \rho)+C\Big(\fint_{B_{ \rho}(x)}|\nabla u(z)-\nabla u(x)|^{\gamma_0}\,dz\Big)^{1/\gamma_0}.
\end{align*}
Therefore, from the definition of $\phi$ and the fact that $0<\gamma_0<1$, we obtain that
\begin{equation}\label{limit}
   \lim_{\rho\rightarrow 0}\mathbf{q}_{x,\rho}=\nabla u(x)
\end{equation}
holds for any Lebesgue point $x\in \Omega$ of the vector-valued function $\nabla u$.

\begin{proposition}\label{prop1}
Suppose that $u\in W_{\text{loc}}^{1,p}(\Omega)$ is a solution to $(\ref{eq:u})$. Then for any $\varepsilon\in(0,1)$ and $B_{2r}(x_0)\subset\subset\Omega$, we have
\begin{equation}\label{ineq:prop}
\begin{aligned}
         \phi(x_0,\varepsilon r)\leq& C\varepsilon^{\alpha}\phi(x_0,r)+C_\varepsilon\left(\frac{|\mu|(B_{2r}(x_0))}{r^{n-1}}\right)^\frac{1}{p-1}\\
         &+C_\varepsilon\frac{|\mu|(B_{2r}(x_0))}{r^{n-1}}
         \fint_{B_{2r}(x_0)}(|\nabla u|+s)^{2-p}\\
         &+C_\varepsilon\omega(r)\left(\fint_{B_{2r}(x_0)}(|\nabla u|+s)^{2-p}\right)^\frac{1}{2-p},
\end{aligned}
\end{equation}
where  $\alpha\in(0,1)$ is the constant in Theorem \ref{thm:osc},  $\gamma_0$ is the same constant as in Lemma \ref{lem:u-w}, $C_\varepsilon$ is a constant depending on $\varepsilon$, $n$, $p$, $\lambda$, and $\gamma_0$, and $C$ is a constant depending on $n$, $p$, $\lambda$, and $\gamma_0$.
\end{proposition}
\begin{proof}
By Theorem \ref{thm:osc} and the definition of $\mathbf{q}_{x,\rho}(\cdot)$, we have
\begin{equation}\label{ineq:phi}
\begin{aligned}
        &\left(\fint_{B_{\varepsilon r}(x_0)}|\nabla u-\mathbf{q}_{x_0,\varepsilon r}(u)|^{\gamma_0}\right)^{\frac 1{\gamma_0}}
        \le \left(\fint_{B_{\varepsilon r}(x_0)}|\nabla u-\mathbf{q}_{x_0,\varepsilon r}(v)|^{\gamma_0}\right)^{\frac 1{\gamma_0}}\\
        &\leq C\left(\fint_{B_{\varepsilon r}(x_0)}|\nabla v-\mathbf{q}_{x_0,\varepsilon r}(v)|^{\gamma_0}\right)^{\frac 1{\gamma_0}}+C\left(\fint_{B_{\varepsilon r}(x_0)}|\nabla u-\nabla v|^{\gamma_0}\right)^{\frac 1{\gamma_0}}\\
        &\leq C\varepsilon^{\alpha}\left(\fint_{B_r(x_0)}|\nabla v-\mathbf{q}_{x_0,r}(v)|^{\gamma_0}\right)^{\frac 1{\gamma_0}}+ C\varepsilon^{-\frac n{\gamma_0}}\left(\fint_{B_r(x_0)}|\nabla u-\nabla v|^{\gamma_0}\right)^{\frac 1{\gamma_0}}\\
        &\leq C\varepsilon^{\alpha}\left(\fint_{B_r(x_0)}|\nabla v-\mathbf{q}_{x_0,r}(u)|^{\gamma_0}\right)^{\frac 1{\gamma_0}}+ C\varepsilon^{-\frac n{\gamma_0}}\left(\fint_{B_r(x_0)}|\nabla u-\nabla v|^{\gamma_0}\right)^{\frac 1{\gamma_0}}\\
        &\leq C\varepsilon^{\alpha}\left(\fint_{B_r(x_0)}|\nabla u-\mathbf{q}_{x_0,r}(u)|^{\gamma_0}\right)^{\frac 1{\gamma_0}}+ C\varepsilon^{-\frac n{\gamma_0}}\left(\fint_{B_r(x_0)}|\nabla u-\nabla v|^{\gamma_0}\right)^{\frac 1{\gamma_0}}.
\end{aligned}
\end{equation}
Moreover, by (\ref{eq:v-w}) and the fact that $|\omega(r)|\leq1$, one has
\begin{equation}\label{ineq:u-v}
    \begin{aligned}
         &\fint_{B_r(x_0)}|\nabla u-\nabla v|^{\gamma_0}\leq \fint_{B_r(x_0)}|\nabla u-\nabla w|^{\gamma_0}+\fint_{B_r(x_0)}|\nabla w-\nabla v|^{\gamma_0}\\
         &\leq C\fint_{B_{2r}(x_0)}|\nabla u-\nabla w|^{\gamma_0}+C\omega(r)^{\gamma_0}\fint_{B_{2r}(x_0)}(|\nabla w|+s)^{\gamma_0}\\
         &\leq C\fint_{B_{2r}(x_0)}|\nabla u-\nabla w|^{\gamma_0}+C\omega(r)^{\gamma_0}\fint_{B_{2r}(x_0)}(|\nabla u|+s)^{\gamma_0}.
    \end{aligned}
\end{equation}
Thus from \eqref{ineq:phi} and \eqref{ineq:u-v}, we have
\begin{equation}\label{ineq:phi2}
\begin{aligned}
\phi(x_0,\epsilon r)&\leq C\varepsilon^{\alpha}\phi(x_0,r)+C_\epsilon \left(\fint_{B_{2r}(x_0)}|\nabla u-\nabla w|^{\gamma_0}\right)^{1/\gamma_0}\\
&\quad +C_\epsilon\omega(r)\left(\fint_{B_{2r}(x_0)}(|\nabla u|+s)^{\gamma_0}\right)^{1/\gamma_0}.
    \end{aligned}
\end{equation}
Now we can apply Lemma \ref{lem:u-w} to bound the second term on the right-hand side of (\ref{ineq:phi2}) to conclude the proof.
\end{proof}
Now we are ready to prove Theorem \ref{thm:int}.

\begin{proof}[Proof of Theorem \ref{thm:int}]
We prove the theorem at a Lebesgue point $x=x_0$ of the vector-valued function $\nabla u$, assuming that $B_R(x_0)\subset\Omega$. Since $p\in \big(\frac{3n-2}{2n-1},2\big)$, we choose $\gamma_0=2-p$ in Lemma \ref{lem:u-w}. Choose $\varepsilon=\varepsilon(n,p,\lambda,\alpha)\in (0,1/4)$ sufficiently small so that $C\varepsilon^\alpha\leq 1/4$, where $C$ is the constant in (\ref{ineq:prop}).

For an integer $j\ge 0$, set $r_j=\varepsilon^jR$, $B^j=B_{2r_j}(x_0)$,
$$
T_j=\left(\fint_{B^j}(|\nabla u|+s)^{2-p}\,dx\right)^\frac{1}{2-p}, \quad \phi_j=\phi(x_0,r_j), \quad \mathbf{q}_j=\mathbf{q}_{x_0,r_j}.
$$
Applying (\ref{ineq:prop}) yields
\begin{equation*}
    \phi_{j+1}\leq\frac{1}{4}\phi_j+C\left(\frac{|\mu|(B^j)}{r_j^{n-1}}\right)^\frac{1}{p-1}
    +C\frac{|\mu|(B^j)}{r_j^{n-1}}T_j^{2-p}+C\omega(r_j) T_j.
\end{equation*}
Let $j_0$ and $m$ be positive integers to be specified later such that $j_0\le m$.
Summing the above inequality over $j={j_0,j_0+1,\dots,m}$, we obtain
\begin{equation}\label{sum:int}
\begin{aligned}
    \sum_{j=j_0}^{m+1}\phi_j&\leq C\phi_{j_0}+C\sum_{j=j_0}^{m}\left(\frac{|\mu|(B^j)}{r_j^{n-1}}\right)^\frac{1}{p-1}\\
   &\quad  +C\sum_{j=j_0}^m \frac{|\mu|(B^j)}{r_j^{n-1}} T_j^{2-p}+C\sum_{j=j_0}^m\omega(r_j)T_j.
\end{aligned}
\end{equation}
Since
\begin{align*}
   |\mathbf{q}_{j+1}-\mathbf{q}_j|^{\gamma_0}\leq |\mathbf{q}_{j+1}-\nabla u(x)|^{\gamma_0}+|\nabla u(x)-\mathbf{q}_{j}|^{\gamma_0},
\end{align*}
by taking the average over $x\in B_{r_{j+1}}(x_0)$ and then taking $\gamma_0$-th root, we obtain
\begin{align*}
    |\mathbf{q}_{j+1}-\mathbf{q}_{j}|\leq C\phi_j+C\phi_{j+1}.
\end{align*}
Then, by iterating, we get
\begin{align*}
     |\mathbf{q}_{m+1}-\mathbf{q}_{j_0}|\leq C \sum_{j=j_0}^{m+1}\phi_j,
\end{align*}
which together with \eqref{sum:int} implies
\begin{equation}\label{ineq:q1}
\begin{aligned}
    |\mathbf{q}_{m+1}|+\sum_{j=j_0}^{m+1}\phi_j\leq& C\phi_{j_0}+|\mathbf{q}_{j_0}|+C\sum_{j=j_0}^m\left(\frac{|\mu|(B^j)}{r_j^{n-1}}\right)^\frac{1}{p-1}\\
    &+C\sum_{j=j_0}^m \frac{|\mu|(B^j)}{r_j^{n-1}} T_j^{2-p}
    +C\sum_{j=j_0}^m\omega(r_j)T_j.
    \end{aligned}
\end{equation}
By the definition of $\phi_{j_0}$, we have
\begin{align*}
    \phi_{j_0}\leq C \left(\fint_{B^{j_0}}|\nabla u|^{\gamma_0}\,dx\right)^\frac{1}{\gamma_0}\leq C T_{j_0}.
\end{align*}
Since
$$
|\mathbf{q}_{j_0}|^{\gamma_0}\leq |\nabla u(x)-\mathbf{q}_{j_0}|^{\gamma_0}+|\nabla u(x)|^{\gamma_0},
$$
by taking the average over $x\in B_{r_{j_0}}(x_0)$ and taking the $\gamma_0$-th root, we obtain
$$
|\mathbf{q}_{j_0}|\leq C\phi_{j_0}+C \left(\fint_{B^{j_0}}|\nabla u|^{\gamma_0}\,dx\right)^\frac{1}{\gamma_0}\leq C T_{j_0}.
$$
Therefore, \eqref{ineq:q1} implies that
\begin{equation}\label{ineq:q}
\begin{aligned}
    |\mathbf{q}_{m+1}|+\sum_{j=j_0}^{m+1}\phi_j\leq& CT_{j_0}+C\sum_{j=j_0}^m\left(\frac{|\mu|(B^j)}{r_j^{n-1}}\right)^\frac{1}{p-1}\\
    &+C\sum_{j=j_0}^m \frac{|\mu|(B^j)}{r_j^{n-1}} T_j^{2-p}
    +C\sum_{j=j_0}^m\omega(r_j)T_j.
    \end{aligned}
\end{equation}

By \eqref{dini} and the comparison principle for Riemann integrals, there exists $j_0=j_0(n,p,\varepsilon,C,\omega)>1$ sufficiently large such that
\begin{equation}\label{ineq:omg}
    4^\frac{1}{\gamma_0}(2\epsilon)^{-\frac{n}{\gamma_0}}C\sum_{j=j_0}^\infty \omega(r_j)\leq\frac{1}{10},
\end{equation}
where $C$ is the constant in (\ref{ineq:q}).

Note that by the comparison principle for Riemann integrals,
\begin{equation}
                                \label{eq7.13}
    \sum_{j=j_0}^m \frac{|\mu|(B^j)}{r_j^{n-1}} \leq C\int_0^{2r_{j_0-1}} \frac{|\mu|(B_\rho(x_0))}{\rho^{n-1}} \frac{d\rho}{\rho}
\end{equation}
and since $p<2$ we also have
\begin{equation}
                                \label{eq7.14}
    \sum^m_{j=j_0}\left(\frac{|\mu|(B^j)}{r_j^{n-1}}\right)^\frac{1}{p-1}\leq C\left(\int_0^{2r_{j_0-1}} \frac{|\mu|(B_\rho(x_0))}{\rho^{n-1}}
    \frac{d\rho}{\rho}\right)^\frac{1}{p-1}.
\end{equation}
To prove \eqref{ineq:int} at $x=x_0$, it is sufficient to show that
\begin{equation}\label{ineq:result}
    |\nabla u(x_0)|\leq CT_{j_0}+ C\left(\int_0^{2r_{j_0-1}}\frac{|\mu|(B_\rho(x_0))}{\rho^{n-1}}
    \frac{d\rho}{\rho}\right)^\frac{1}{p-1}.
\end{equation}
To this end, we consider the following possibilities.

{\em Case 1:} If $|\nabla u(x_0)|\leq T_{j_0}$, then (\ref{ineq:result}) easily follows.

{\em Case 2:} If $ T_j< |\nabla u(x_0)|, \ \forall j_0\leq j\leq j_1$, and $|\nabla u(x_0)|\leq T_{j_1+1}$,
then since $\gamma_0=2-p<1$, we have
\begin{equation}\label{ineq:grad}
\begin{aligned}
       |\nabla u(x_0)|&\leq\left(\fint_{B^{j_1+1}}(|\nabla u|+s)^{\gamma_0}\,dx\right)^{1/\gamma_0}\\
       &\leq 2^{\frac{1}{\gamma_0}}\left(\fint_{B^{j_1+1}}|\nabla u|^{\gamma_0}\,dx\right)^{1/\gamma_0}+2^{\frac{1}{\gamma_0}}s\\
       &\leq 2^{\frac{1}{\gamma_0}}(2\epsilon)^{-\frac{n}{\gamma_0}}
       \left(\fint_{B_{r_{j_1}}(x_0)}|\nabla u|^{\gamma_0}\,dx\right)^{1/\gamma_0}+2^{\frac{1}{\gamma_0}}s\\
       &\leq 4^\frac{1}{\gamma_0}(2\epsilon)^{-\frac{n}{\gamma_0}}(\phi_{j_1}+|\mathbf{q}_{j_1}|)
       +2^{\frac{1}{\gamma_0}}s,
\end{aligned}
\end{equation}
where the last inequality follows from the definitions of $\phi_{j_1}$ and $\mathbf{q}_{j_1}$.
Now applying (\ref{ineq:q}) with $m=j_1-1$ and using \eqref{eq7.13} and \eqref{eq7.14}, from \eqref{ineq:grad} we get
\begin{align*}
|\nabla u(x_0)|\leq&C'T_{j_0}
+C{''}\left(\int_0^{2r_{j_0-1}}\frac{|\mu|(B_\rho(x_0))}{\rho^{n-1}}
\frac{d\rho}{\rho}\right)^\frac{1}{p-1}\\
&+C{''}\int_0^{2r_{j_0-1}}
\frac{|\mu|(B_\rho(x_0))}{\rho^{n-1}}\frac{d\rho}{\rho}\cdot|\nabla u(x_0)|^{2-p}\\
&+C'\sum_{j=j_0}^{m}\omega(r_j)|\nabla u(x_0)|+2^\frac{1}{\gamma_0}s,
\end{align*}
where $C'=4^\frac{1}{\gamma_0}(2\epsilon)^{-\frac{n}{\gamma_0}}C$, $C$ is the constant in (\ref{ineq:q}), and $C''$ is a constant depending on $n$, $p$, and $\lambda$.
Hence using (\ref{ineq:omg}) and Young's inequality, we find
\begin{align*}
    |\nabla u(x_0)|\leq CT_{j_0}+C \left(\int_0^{2r_{j_0-1}}
    \frac{|\mu|(B_\rho(x_0))}{\rho^{n-1}}\frac{d\rho}{\rho}\right)^\frac{1}{p-1}
    +\frac{1}{5}|\nabla u(x_0)|+Cs.
\end{align*}
This implies (\ref{ineq:result}) as desired.

{\em Case 3:} If $T_j<|\nabla u(x_0)|$ for any $j\geq j_0$, then from \eqref{ineq:q}, \eqref{eq7.13}, and \eqref{eq7.14} we have for any $m>j_0$,
\begin{align*}
    |\mathbf{q}_{m+1}|\leq& CT_{j_0}+C\left(\int_0^{2r_{j_0-1}}
    \frac{|\mu|(B_\rho(x_0))}{\rho^{n-1}}\frac{d\rho}{\rho}\right)^\frac{1}{p-1}\\
    &+C\int_0^{2r_{j_0-1}}
    \frac{|\mu|(B_\rho(x_0))}{\rho^{n-1}}\frac{d\rho}{\rho}\cdot|\nabla u(x_0)|^{2-p}+C\sum_{j=j_0}^m\omega(r_j)|\nabla u(x_0)|\\
    \leq& CT_{j_0}+C\left(\int_0^{2r_{j_0-1}}
    \frac{|\mu|(B_\rho(x_0))}{\rho^{n-1}}\frac{d\rho}{\rho}\right)^\frac{1}{p-1}\\
    &+C\int_0^{2r_{j_0-1}}
    \frac{|\mu|(B_\rho(x_0))}{\rho^{n-1}}\frac{d\rho}{\rho}\cdot|\nabla u(x_0)|^{2-p}+\frac{1}{10}|\nabla u(x_0)|.
\end{align*}
Here we used (\ref{ineq:omg}) in the last inequality. Letting $m\to \infty$ and using \eqref{limit}, we get
\begin{align*}
    |\nabla u(x_0)|\leq& CT_{j_0}+C\left(\int_0^{2r_{j_0-1}}
    \frac{|\mu|(B_\rho(x_0))}{\rho^{n-1}}\frac{d\rho}{\rho}\right)^\frac{1}{p-1}\\
    &+C\int_0^{2r_{j_0-1}}
    \frac{|\mu|(B_\rho(x_0))}{\rho^{n-1}}\frac{d\rho}{\rho}\cdot|\nabla u(x_0)|^{2-p}+\frac{1}{10}|\nabla u(x_0)|.
\end{align*}
Then using Young's inequality, we deduce (\ref{ineq:result}). The proof is completed.
\end{proof}

\section{Interior Lipschitz estimate and modulus of continuity estimate of the gradient}\label{sec4}
In this section, we give the proof of the Lipschitz estimate in Theorem \ref{thm:int:lip} and derive an interior modulus of continuity estimate of $\nabla u$ under the same conditions.
We first adapt the argument in \cite{dong2017c1} to obtain some decay estimates from Proposition \ref{prop1}.
Let $\alpha\in(0,1)$ be the same constant as in Theorem \ref{thm:osc}, ${\alpha_1}\in (0,\alpha)$,  $R\in (0, 1]$ and $B_{R}(x_0)\subset \Omega$. Choose $\varepsilon=\epsilon(n,p,\lambda,\gamma_0,\alpha,\alpha_1)>0$ sufficiently small such that
$$
C\varepsilon^{\alpha-{\alpha_1}}<1\quad \text{and}\quad
\varepsilon^{\alpha_1}<1/4,
$$
where $C$ is the constant in \eqref{ineq:prop}.

Proposition \ref{prop1} implies that for any $B_{2r}(x)\subset\subset B_R(x_0)$,
\begin{equation}
                        \label{eq7.15}
    \begin{aligned}
         \phi(x,\varepsilon r)
         \leq&\varepsilon^{\alpha_1}\phi(x,r)+C\left(\frac{|\mu|(B_{2r}(x))}{r^{n-1}}\right)^\frac{1}{p-1}\\
         &+C\frac{|\mu|(B_{2r}(x))}{r^{n-1}}\left(\|\nabla u\|_{L^\infty(B_{2r}(x))}+s\right)^{2-p}\\
         &+C\omega(r)\left(\|\nabla u\|_{L^\infty(B_{2r}(x))}+s\right).
\end{aligned}
\end{equation}
Denote
\begin{equation}\label{def:gh}
g(x,r)=\frac{|\mu|(B_{r}(x))}{r^{n-1}},\quad h(x,r)=g(x,r)^\frac{1}{p-1}.
\end{equation}
By iteration, from \eqref{eq7.15} we get
\begin{align*}
    \phi(x,\varepsilon^j r)\leq &\varepsilon^{{\alpha_1} j}\phi(x,r)+C\sum_{i=1}^j \varepsilon^{{\alpha_1} (i-1)}h(x,2\varepsilon^{j-i}r)\\
    &+C\sum_{i=1}^j \varepsilon^{{\alpha_1}(i-1)}g(x,2\varepsilon^{j-i}r)\left(\|\nabla u\|_{L^\infty(B_{2r}(x))}+s\right)^{2-p}\\
    &+C\sum_{i=1}^j \varepsilon^{{\alpha_1} (i-1)}\omega(\varepsilon^{j-i}r)\left(\|\nabla u\|_{L^\infty(B_{2r}(x))}+s\right)
\end{align*}
for any $B_{2r}(x)\subset\subset B_R(x_0)$ with $r\in (0,R/4)$. Thus,
\begin{align}\label{iter:int1}
    \phi(x,\varepsilon^jr)\leq &\varepsilon^{{\alpha_1} j}\phi(x,r)+C\Tilde{h}(x,2\varepsilon^j r)+C\Tilde{g}(x,2\varepsilon^j r)\left(\|\nabla u\|_{L^\infty(B_{2r}(x))}+s\right)^{2-p}\notag
    \\&+C\Tilde{\omega}(\varepsilon^j r)\left(\|\nabla u\|_{L^\infty(B_{2r}(x))}+s\right),
\end{align}
where we defined
\begin{equation}
                \label{eq6.42}
\begin{aligned}
    &\Tilde{h}(x,t):=\sum_{i=1}^\infty \varepsilon^{{\alpha_1} i}\left(h(x,\varepsilon^{-i}t)[\epsilon^{-i}t\leq R/2]+h(x,R/2)[\varepsilon^{-i}t>R/2]\right),\\
    &\Tilde{g}(x,t):=\sum_{i=1}^\infty \varepsilon^{{\alpha_1} i}\left(g(x,\varepsilon^{-i}t)[\epsilon^{-i}t\leq R/2]+g(x,R/2)[\varepsilon^{-i}t>R/2]\right),\\
    &\Tilde{\omega}(t):=\sum_{i=1}^\infty \varepsilon^{{\alpha_1} i}\left(\omega(\varepsilon^{-i}t)[\epsilon^{-i}t\leq R/2]+\omega(R/2)[\varepsilon^{-i}t>R/2]\right).
\end{aligned}
\end{equation}
Here we used the Iverson bracket notation, i.e., $[P] = 1$ if $P$ is true and $[P] = 0$ otherwise. We obtain the following lemma from \eqref{iter:int1}.
\begin{lemma}\label{lem:iter1}
Let $B_{2r}(x)\subset\subset B_R(x_0)\subset \Omega$ with $r\le R/4$. There exists a constant $C$ depending only on $\varepsilon$, $n$, $p$, $\lambda$, $\gamma_0$, and $\alpha_1$, such that for any $\rho\in (0,r]$, we have
\begin{itemize}
    \item[(i)]
\begin{equation}\label{iter:int2}
\begin{aligned}
    \phi(x,\rho)&\leq  C\left(\frac{\rho}{r}\right)^{\alpha_1} \phi(x,r)+C\Tilde{h}(x,2\rho)
    \\&\,\,+C\Tilde{g}(x,2\rho)\left(\|\nabla u\|_{L^\infty(B_{2r}(x))}+s\right)^{2-p}
  +C\Tilde{\omega}(\rho)\left(\|\nabla u\|_{L^\infty(B_{2r}(x))}+s\right),
\end{aligned}
\end{equation}
   \item[(ii)]
  \begin{equation}\label{sum1:phi}
\begin{aligned}
    \sum_{j=0}^\infty \phi(x,\varepsilon^j\rho)
    \leq& C\left(\frac{\rho}{r}\right)^{\alpha_1}\phi(x,r)
    +C\int_0^{\rho}\frac{\Tilde{h}(x,t)}{t}\,dt
    \\&+C\left(\|\nabla u\|_{L^\infty(B_{2r}(x))}+s\right)^{2-p}\int_0^{\rho}\frac{\Tilde{g}(x,t)}{t}\,dt\\&
    +C\left(\|\nabla u\|_{L^\infty(B_{2r}(x))}+s\right)\int_0^{\rho}\frac{\Tilde{\omega}(t)}{t}\,dt.
\end{aligned}
\end{equation}
\end{itemize}
\end{lemma}

To prove Lemma \ref{lem:iter1}, we need the following technical lemma.
\begin{lemma}\label{lem:t}
Let $B_R(x_0)\subset \Omega$. Then there exist constants $c_1$ and $c_2$ depending on $\varepsilon$, $n$, $p$, and ${\alpha_1}$, such that for any fixed $x\in B_{R}(x_0)$ and any $f\in\{\Tilde{\omega}, \Tilde{g}(x,\cdot), \Tilde{h}(x,\cdot) \}$, it holds that
    $c_1 f(t)\leq f(s)\leq c_2 f(t)$, whenever $0<\varepsilon t\leq s\leq t$.
\end{lemma}
\begin{proof}
We will only show the proof for $g$ since the other cases are similar. For fixed $x\in B_{R/4}(x_0)$, we set
\begin{equation*}
   G(x,r)= \left\{
\begin{aligned}
    &g(x,r)\; \qquad  \text{if} \quad 0<r\leq R/2, \\
    &g(x,R/2) \quad \text{if} \quad r>R/2,
\end{aligned}
\right.
\end{equation*}
and observe that by \eqref{eq6.42},
$$
\Tilde{g}(x,r)=\sum_{i=1}^\infty \varepsilon^{{\alpha_1} i}G(x,\varepsilon^{-i}r).
$$
Suppose that $0<\varepsilon t\leq s\leq t$. It is easy to see from the definitions of $g$ and $G$ that $G(x,s)\leq \epsilon^{1-n} G(x,t)$ and therefore $\Tilde{g}(x,s)\leq \epsilon^{1-n} \Tilde{g}(x,t)$. Also the fact that $0<\epsilon s\leq \epsilon t \leq s$ implies that $\Tilde{g}(x,\epsilon t)\leq \epsilon^{1-n} \Tilde{g}(x,s)$. On the other hand,
$$
\Tilde{g}(x,t)=\epsilon^{-{\alpha_1}}\sum_{i=1}^\infty \varepsilon^{{\alpha_1}( i+1)}G(x,\varepsilon^{-(i+1)}\epsilon t)\leq \epsilon^{-{\alpha_1}}\Tilde{g}(x,\epsilon t).
$$
The lemma is proved.
\end{proof}

Now we are ready to prove Lemma \ref{lem:iter1}.
\begin{proof}[Proof of Lemma \ref{lem:iter1}]
We first prove Assertion (i). For given $\rho\in(0,r]$, let $j$ be an integer such that
$\epsilon^{j+1}<{\rho}/{r}\leq \epsilon^j$. Then by \eqref{iter:int1} with $\epsilon^{-j}\rho$ in place of $r$, we get
\begin{align*}
      \phi(x,\rho)\leq &\varepsilon^{{\alpha_1} j}\phi(x,\epsilon^{-j}\rho)+C\Tilde{h}(x,2 \rho)+C\Tilde{g}(x,2\rho)\left(\|\nabla u\|_{L^\infty(B_{2\epsilon^{-j}\rho}(x))}+s\right)^{2-p}
    \\&+C\Tilde{\omega}( \rho)\left(\|\nabla u\|_{L^\infty(B_{2\epsilon^{-j}\rho}(x))}+s\right)\\
    \leq & C \left(\frac{\rho}{r}\right)^{\alpha_1} \phi(x,r)+C\Tilde{h}(x,2\rho)
    +C\Tilde{g}(x,2\rho)\left(\|\nabla u\|_{L^\infty(B_{2r}(x))}+s\right)^{2-p}
    \\&+C\Tilde{\omega}(\rho)\left(\|\nabla u\|_{L^\infty(B_{2r}(x))}+s\right).
\end{align*}
Therefore, Assertion (i) holds. Now applying \eqref{iter:int2} with $\epsilon^j \rho$ in place of $\rho$ and summing in $j$, we get
\begin{align*}
    \sum_{j=0}^\infty \phi(x,\varepsilon^j\rho)
    \leq &C\left(\frac{\rho}{r}\right)^{\alpha_1}\phi(x,r)
    +C\sum_{j=1}^\infty \Tilde{h}(x,2\varepsilon^j \rho)\\
    &+C\left(\|\nabla u\|_{L^\infty(B_{2r}(x))}+s\right)^{2-p}\sum_{j=1}^\infty \Tilde{g}(x,2\varepsilon^j\rho)\\&
    +C\left(\|\nabla u\|_{L^\infty(B_{2r}(x))}+s\right)\sum_{j=1}^\infty\Tilde{\omega}(\varepsilon^j\rho).
\end{align*}
Hence, by using Lemma \ref{lem:t} and the comparison principle for Riemann integrals, we can easily get \eqref{sum1:phi}.
The lemma is proved.
\end{proof}

Recall the definition of $\mathbf{q}_{x,\rho}$ from Section \ref{sec3}.
Since we have
\begin{align*}
    |\mathbf{q}_{x,\varepsilon \rho}-\mathbf{q}_{x,\rho}|^{\gamma_0}\leq|\nabla u(z)-\mathbf{q}_{x,\rho}|^{\gamma_0}+|\nabla u(z)-\mathbf{q}_{x,\varepsilon \rho}|^{\gamma_0},
\end{align*}
by taking the average over $z\in B_{\varepsilon \rho}(x)$ and then taking the $\gamma_0$-th root, we obtain
\begin{align*}
    |\mathbf{q}_{x,\varepsilon \rho}-\mathbf{q}_{x,\rho}|\leq C\phi(x,\varepsilon \rho)+C\phi(x,\rho).
\end{align*}
Then, by iterating, we get
\begin{align*}
    |\mathbf{q}_{x,\varepsilon^j\rho}-\mathbf{q}_{x,\rho}|\leq C\sum_{i=0}^j\phi(x,\varepsilon^i\rho).
\end{align*}
Therefore, by using \eqref{limit}, we obtain that
\begin{equation}\label{ineq:diff}
     |\nabla u(x)-\mathbf{q}_{x,\rho}|\leq C\sum_{j=0}^\infty \phi(x,\varepsilon^j\rho)
\end{equation}
holds for any Lebesgue point $x\in \Omega$ of the vector-valued function $\nabla u$.

Now we are ready to prove the interior Lipschitz estimate.
\begin{proof}[Proof of Theorem \ref{thm:int:lip}]
We prove the theorem around a given point $x=x_0$ assuming that $B_R(x_0)\subset\Omega$ with $R\in (0,1]$ and
\begin{equation}\label{finite Riesz}
    \big\|\mathbf{I}_1^R(|\mu|)\big\|_{L^\infty(B_{R}(x_0))}<\infty.
\end{equation}
We first derive an a priori estimate for the case when $u\in C^1$ and then use approximation to prove the general case.

{\em Step 1: The case when $u\in C^1\big(\,\overline{B_R(x_0)}\,\big)$.}

Using \eqref{sum1:phi} with $\rho=r$ and \eqref{ineq:diff}, we obtain that
\begin{equation*}
\begin{aligned}
   |\nabla u(x)-\mathbf{q}_{x,\rho}|
    \leq& C\phi(x,\rho)
    +C\int_0^{\rho}\frac{\Tilde{h}(x,t)}{t}\,dt
    \\&+C\left(\|\nabla u\|_{L^\infty(B_{2\rho}(x))}+s\right)^{2-p}\int_0^{\rho}\frac{\Tilde{g}(x,t)}{t}\,dt\\&
    +C\left(\|\nabla u\|_{L^\infty(B_{2\rho}(x))}+s\right)\int_0^{\rho}\frac{\Tilde{\omega}(t)}{t}\,dt
\end{aligned}
\end{equation*}
holds for any $B_{2\rho}(x)\subset\subset B_R(x_0)$ with $\rho\in (0, R/4]$.

Note that
$$
|\mathbf{q}_{x,\rho}|\leq C \phi(x,\rho)+C \rho^{-n/\gamma_0}\|\nabla u\|_{L^{\gamma_0}(B_{\rho}(x))}\leq C\rho^{-n/\gamma_0}\|\nabla u\|_{L^{\gamma_0}(B_{\rho}(x))}.
$$
Combining the above two inequalities, we have
\begin{equation}\label{sum1:phi_1}
\begin{aligned}
   |\nabla u(x)|
    \leq& C \rho^{-n/\gamma_0}\|\nabla u\|_{L^{\gamma_0}(B_{\rho}(x))}
    +C\int_0^{\rho}\frac{\Tilde{h}(x,t)}{t}\,dt
    \\&+C\left(\|\nabla u\|_{L^\infty(B_{2\rho}(x))}+s\right)^{2-p}\int_0^{\rho}\frac{\Tilde{g}(x,t)}{t}\,dt\\&
    +C\left(\|\nabla u\|_{L^\infty(B_{2\rho}(x))}+s\right)\int_0^{\rho}\frac{\Tilde{\omega}(t)}{t}\,dt.
\end{aligned}
\end{equation}
Note that $\Tilde{\omega}$ also satisfies the Dini condition \eqref{dini}; see \cite[Lemma 1]{dong2012gradient}. Thus we can take $\rho_0=\rho_0(n,p,\lambda,\omega,\alpha_1,\gamma_0,R)\in(0,R/4]$ sufficiently small such that
$$
C \int_0^{\rho_0}\frac{\Tilde{\omega}(t)}{t}\,dt\leq 3^{-1-n/\gamma_0},
$$
where $C$ is the constant in \eqref{sum1:phi_1}.
Then for any $B_{2\rho}(x)\subset\subset B_R(x_0)$ with $0<\rho\leq \rho_0$, by \eqref{sum1:phi_1} and Young's inequality, we have
\begin{equation}\label{sum1:phi_2}
\begin{aligned}
   |\nabla u(x)|
    \leq& C \rho^{-n/\gamma_0}\|\nabla u\|_{L^{\gamma_0}(B_{\rho}(x))}
    +C\int_0^{\rho}\frac{\Tilde{h}(x,t)}{t}\,dt
    \\&+C\left(\int_0^{\rho}\frac{\Tilde{g}(x,t)}{t}\,dt\right)^{\frac{1}{p-1}}
    + 3^{-n/\gamma_0}\left(\|\nabla u\|_{L^\infty(B_{2\rho}(x))}+s\right).
\end{aligned}
\end{equation}
For $k\geq 1$, we denote $\rho_k=(1-2^{-k})R$. Since $\rho_{k+1}-\rho_k=2^{-k-1}R$, we have $B_{2\rho}(x)\subset B_{\rho_{k+1}}(x_0)$ for any $x\in  B_{\rho_{k}}(x_0)$ and $\rho=2^{-k-2}R$. We take $k_0$ sufficiently large such that $2^{-k_0-2}\leq \rho_0$. Then by \eqref{sum1:phi_2} with $\rho=2^{-k-2}R$, we have for any $k\geq k_0$ that
\begin{align*}
    &\|\nabla u\|_{L^\infty(B_{\rho_k}(x_0))}+s
    \\&\leq 3^{-n/{\gamma_0}}(\|\nabla u\|_{L^\infty(B_{\rho_{k+1}}(x_0))}+s)+C \left(\frac{2^{k+2}}{R}\right)^{n/\gamma_0}\|\nabla u\|_{L^{\gamma_0}(B_{\rho_{k+1}}(x_0))}
\\&\quad+C \sup_{x\in B_{\rho_{k}}(x_0)}\int_0^{\frac{R}{2}}\frac{\Tilde{h}(x,t)}{t}\,dt
+C \sup_{x\in B_{\rho_k}(x_0)} \left(\int_0^{\frac{R}{2}}\frac{\Tilde{g}(x,t)}{t}\,dt\right)^{\frac{1}{p-1}}+s
\end{align*}
Multiplying the above inequality by $3^{-nk/{\gamma_0}}$, and summing the terms with respect to $k=k_0,k_0+1,\cdots$, we obtain that
\begin{align*}
    &\sum_{k=k_0}^\infty 3^{-nk/{\gamma_0}}( \|\nabla u\|_{L^\infty(B_{\rho_k}(x_0))}+s)\\&\leq  \sum_{k=k_0+1}^\infty 3^{-nk/{\gamma_0}}( \|\nabla u\|_{L^\infty(B_{\rho_k}(x_0))}+s)
    +CR^{-n/{\gamma_0}} (\|\nabla u\|_{L^{\gamma_0}(B_{R}(x_0))}+s)
    \\&\quad+C   \sup_{x\in B_{R}(x_0)}\int_0^{\frac{R}{2}}\frac{\Tilde{h}(x,t)}{t}\,dt
    +C \sup_{x\in B_{R}(x_0)} \left(\int_0^{\frac{R}{2}}\frac{\Tilde{g}(x,t)}{t}\,dt\right)^{\frac{1}{p-1}}+Cs,
\end{align*}
where each summation is finite. By subtracting
$$
\sum_{k=k_0+1}^\infty 3^{-nk/{\gamma_0}}( \|\nabla u\|_{L^\infty(B_{\rho_k}(x_0))}+s)
$$
from both sides of the above inequality, we get the following $L^{\infty}$-estimate for $\nabla u$:

\begin{equation}\label{infty-norm1}
    \begin{aligned}
         &\|\nabla u\|_{L^\infty(B_{R/2}(x_0))}+s\leq C R^{-n/{\gamma_0}} \|\nabla u\|_{L^{\gamma_0}(B_{R}(x_0))}
   \\ &\qquad+C   \sup_{x\in B_{R}(x_0)}\int_0^{\frac{R}{2}}\frac{\Tilde{h}(x,t)}{t}\,dt
    +C \sup_{x\in B_{R}(x_0)} \left(\int_0^{\frac{R}{2}}\frac{\Tilde{g}(x,t)}{t}\,dt\right)^{\frac{1}{p-1}}+Cs,
    \end{aligned}
\end{equation}
We can simplify the terms in \eqref{infty-norm1} to get
\begin{equation}\label{infty-norm2}
   \|\nabla u\|_{L^\infty(B_{R/2}(x_0))}
         \leq  C\big\|\mathbf{I}_1^R(|\mu|)\big\|^\frac{1}{p-1}_{L^\infty(B_{R}(x_0))}+CR^{-\frac{n}{2-p}} \||\nabla u|+s\|_{L^{2-p}(B_{R}(x_0))}.
\end{equation}
Indeed, by the definition of $\Tilde{g}$ in \eqref{eq6.42}, we have
\begin{align*}
    \int_0^{R/2}\frac{\Tilde{g}(x,t)}{t}\,dt&=\sum_{i=1}^\infty \varepsilon^{{\alpha_1} i}\int_0^{R/2}\frac{g(x,\epsilon^{-i}t)}{t}[\epsilon^{-i}t\leq R/2]\,dt\\
    &\quad +\sum_{i=1}^\infty \varepsilon^{{\alpha_1} i}\int_0^{R/2}\frac{g(x,R/2)}{t}[\epsilon^{-i}t>R/2]\,dt.
\end{align*}
The first term above is equal to
\begin{align*}
   \sum_{i=1}^\infty\varepsilon^{{\alpha_1} i}\int_0^{\epsilon^i R/2}\frac{g(x,\varepsilon^{-i}t)}{t}\,dt
    =\sum_{i=1}^\infty \varepsilon^{{\alpha_1} i}\int_0^{R/2}\frac{g(x,t)}{t}\,dt
    \leq C \mathbf{I}_1^R(|\mu|)(x).
\end{align*}
The second term is  equal to
\begin{align*}
   \sum_{i=1}^\infty\varepsilon^{{\alpha_1} i}\ln(\varepsilon^{-i})g(x,R/2)
    \leq C g(x,R/2)\leq C \mathbf{I}_1^R(|\mu|)(x).
\end{align*}
Therefore,
\begin{equation}\label{g3}
    \int_0^{{R/2}}\frac{\Tilde{g}(x,t)}{t}\,dt\leq C\mathbf{I}_1^{{R}}(|\mu|)(x).
\end{equation}
We can similarly get
\begin{equation}
\begin{aligned}\label{h3}
   &\int_0^{R/2}\frac{\Tilde{h}(x,t)}{t}\,dt=\sum_{i=1}^\infty \varepsilon^{{\alpha_1} i}\int_0^{R/2}\frac{h(x,t)}{t}\,dt+ \sum_{i=1}^\infty\varepsilon^{{\alpha_1} i}\ln(\varepsilon^{-i})h(x,R/2)
    \\\qquad&\leq C \Big( \int_0^{R}\frac{g(x,t)}{t}\,dt\Big)^\frac{1}{p-1}+C \left(\mathbf{I}_1^{{R}}(|\mu|)(x)\right)^\frac{1}{p-1}
    \leq C \left(\mathbf{I}_1^{{R}}(|\mu|)(x)\right)^\frac{1}{p-1}.
\end{aligned}
\end{equation}
Recalling the fact that $\gamma_0\leq 2-p$ (cf. Lemma \ref{lem:u-w}), using (\ref{g3}), (\ref{h3}), and H\"older's inequality, from (\ref{infty-norm1}) we obtain \eqref{infty-norm2}.

{\em Step 2: The general case.}

We take $r_1\in(0,R)$, $r_2=(R+r_1)/2$, and a sequence of standard mollifiers $\{\varphi_k\}$ such that for any positive integer $k$,
$$
\varphi_k\in C_0^\infty(B_{1/k}(0)),\quad
\varphi_k\geq 0,\quad \text{and}\,
\fint_{B_{1/k}(0)}\varphi_k=1.
$$
Then we mollify $\mu$ and $A$ by setting
$$\mu_k(x)=(\mu*\varphi_k)(x), \quad A_k(x,\xi)=(A(\cdot,\xi)*\varphi_k)(x), \quad x\in B_{r_2}(x_0).$$
We note that $A_k$ is well defined and satisfies the growth, ellipticity and continuity assumptions \eqref{ineq:growth}-\eqref{ineq:osi} in $B_{r_2}(x_0)$ for $k>1/(R-r_2)$.
By the corollary after \cite[Theorem 1]{hedberg1983thin}, \eqref{finite Riesz} implies $\mu\in W^{-1,p'}(B_{r_2}(x_0))$, where $p'=p/(p-1)$, and therefore
\begin{equation}\label{conv1}
 \|\mu_k-\mu\|_{W^{-1,p'}(B_{r_2}(x_0))}\to 0.
\end{equation}
Next we let $u_k\in u+W^{1,p}_0(B_{r_2}(x_0))$ be the unique solution to
\begin{equation}\label{eq:uk}
    \left\{
\begin{aligned}
     -\div(A_k(x,\nabla u_k)) =&  \mu_k\quad \text{in} \quad B_{r_2}(x_0), &\\
     u_k =&  u \quad \text{on}  \quad \partial B_{r_2}(x_0).  &\\
\end{aligned}
\right.
\end{equation}
Choosing $u_k-u$ as a test function in \eqref{eq:uk}, we obtain

\begin{equation}\label{test3}
\begin{aligned}
    &\int_{B_{r_2}(x_0)}\langle A(x,\nabla u_k),\nabla u_k\rangle dx\\
    &=\int_{B_{r_2}(x_0)}\langle A(x,\nabla u_k),\nabla u\rangle \,dx
    +\int_{B_{r_2}(x_0)}(u_k-u)\,d \mu_k.
\end{aligned}
\end{equation}
Using the fundamental theorem of calculus, \eqref{ineq:growth}, \eqref{ineq:elliptic}, and Young's inequality with exponents $p$ and $p/(p-1)$, we have
\begin{align*}
   &\int_{B_{r_2}(x_0)}\langle A(x,\nabla u_k),\nabla u_k\rangle \,dx\\
   &=\int_{B_{r_2}(x_0)}  \Big(\int_0^1 \langle D_\xi A(x,t\nabla u_k)\nabla u_k, \nabla u_k\rangle\; dt+\langle A(x,0),\nabla u_k\rangle\Big)\,dx\\
    &\geq \int_{B_{r_2}(x_0)}  \Big(\int_0^1 \lambda^{-1}(s^2+|\nabla u_k|^2t^2)^\frac{p-2}{2} |\nabla u_k|^2\,dt-\lambda s^{p-1}|\nabla u_k|\Big)\,dx\\
&\geq \lambda^{-1} \int_{B_{r_2}(x_0)}  (s^2+|\nabla u_k|^2)^\frac{p-2}{2} |\nabla u_k|^2 \,dx-\int_{B_{r_2}(x_0)} \lambda s^{p-1}|\nabla u_k|\,dx\\
&\geq c(\lambda, p) \int_{B_{r_2}(x_0)}  |\nabla u_k|^p \,dx-C'(\lambda,p) \int_{B_{r_2}(x_0)}  s^p \,dx .
\end{align*}
On the other hand, using \eqref{ineq:growth} and Young's inequality, we obtain
\begin{align*}
   &\int_{B_{r_2}(x_0)}\langle A(x,\nabla u_k),\nabla u\rangle \,dx +\int_{B_{r_2}(x_0)}(u_k-u)\,d \mu_k \\
   &\leq \lambda\int_{B_{r_2}(x_0)}(s^2+|\nabla u_k|^2)^\frac{p-1}{2}|\nabla u|\,dx+\|u_k-u\|_{W^{1,p}_0(B_{r_2}(x_0))}\|\mu_k\|_{W^{-1,p'}(B_{r_2}(x_0))}\\
   &\leq \frac{1}{4}c(\lambda, p)\int_{B_{r_2}(x_0)}  (|\nabla u_k|+s)^p \,dx+C''(\lambda, p)\int_{B_{r_2}(x_0)}  |\nabla u|^p \,dx\\
   &\quad +\frac{1}{8}c(\lambda, p)\|\nabla u_k-\nabla u\|_{L^p(B_{r_2}(x_0))}^p +C''(\lambda,p)\;\|\mu_k\|_{W^{-1,p'}(B_{r_2}(x_0))}^{p'}.
\end{align*}
Therefore, \eqref{test3} implies that
\begin{equation}\label{bound1}
\|\nabla u_k\|_{L^p(B_{r_2}(x_0))}^p\leq C \||\nabla u|+s\|_{L^p(B_{r_2}(x_0))}^p+C\|\mu_k\|_{W^{-1,p'}(B_{r_2}(x_0))}^{p'},
\end{equation}
where $C$ is a constant not depending on $k$.

Now we recall a well-known inequality
\begin{equation}\label{V}
    c^{-1} \big(s^2+|\xi_1|^2+|\xi_2|^2)^\frac{p-2}{2}\leq \frac{|V(\xi_2)-V(\xi_1)|^2}{|\xi_2-\xi_1|^2}\leq c \big(s^2+|\xi_1|^2+|\xi_2|^2)^\frac{p-2}{2},
\end{equation}
where $c=c(n,p)>1$ is a positive constant and the mapping $V(\cdot)$ is defined as
$$
V(\xi)=(|\xi|^2+s^2)^{\frac{p-2}{4}}\xi, \quad \xi\in \R^n.
$$
Combining \eqref{ineq:elliptic} and \eqref{V} yields
$$
c_0^{-1}|V(\xi_2)-V(\xi_1)|^2 \leq \langle A(x,\xi_2)-A(x,\xi_1),\,\xi_2-\xi_1 \rangle,
$$
for some positive constant $c_0=c_0(n,p,\lambda)$. We note that the above inequality also holds for $A_k$.
Then choosing $(u_k-u)1_{B_{r_2}(x_0)}$ as a test function in \eqref{eq:u} and \eqref{eq:uk}, we have
\begin{align*}
   &c_0^{-1}\int_{B_{r_2}(x_0)}|V(\nabla u_k)-V(\nabla u)|^2\,dx
   \\&\leq\int_{B_{r_2}(x_0)}\langle A_k(x,\nabla u_k)-A_k(x,\nabla u),\nabla u_k-\nabla u\rangle \,dx
   \\&=\int_{B_{r_2}(x_0)}\langle A(x,\nabla u)-A_k(x,\nabla u),\nabla u_k-\nabla u\rangle \,dx +\int_{B_{r_2}(x_0)}(u_k-u)\,d (\mu_k-\mu)
   \\&\leq \|A(\cdot,\nabla u)-A_k(\cdot,\nabla u)\|_{L^{p/(p-1)}(B_{r_2}(x_0))}\cdot\|u_k-u\|_{W^{1,p}_0(B_{r_2}(x_0))}
   \\&\quad+\|u_k-u\|_{W^{1,p}_0(B_{r_2}(x_0))}\cdot\|\mu_k-\mu\|_{W^{-1,p'}(B_{r_2}(x_0))}.
\end{align*}
From the definition of $A_k$, by using the Minkowski inequality and \eqref{ineq:osi}, we obtain
$$
\|A(\cdot,\nabla u)-A_k(\cdot,\nabla u)\|_{L^{p/(p-1)}(B_{r_2}(x_0))}\leq \lambda \omega(1/k)\Big(\int_{B_{r_2}(x_0)}(s^2+|\nabla u|^2)^{\frac{p}{2}}\,dx\Big)^\frac{p-1}{p},
$$
which together with \eqref{conv1} and \eqref{bound1}, yields
$$
\int_{B_{r_2}(x_0)}|V(\nabla u_k)-V(\nabla u)|^2\,dx\to 0.
$$
By \eqref{V}, we have
$$
|\nabla u_k-\nabla u|^p\leq c|V(\nabla u_k)-V(\nabla u)|^p\big(|\nabla u_k|^2+|\nabla u|^2+s^2\big)^\frac{p(2-p)}{4}
$$
and therefore using H\"older's inequality with exponents $2/p$ and $2/(2-p)$, we obtain that
\begin{align*}
  &\int_{B_{r_2}(x_0)}|\nabla u_k-\nabla u|^p\,dx\\&\leq c\Big(\int_{B_{r_2}(x_0)}|V(\nabla u_k)-V(\nabla u)|^2\,dx\Big)^\frac{p}{2}\Big(\int_{B_{r_2}(x_0)}\big(|\nabla u_k|^2+|\nabla u|^2+s^2\big)^\frac{p}{2}\,dx\Big)^\frac{2-p}{2},
\end{align*}
which implies that
$$\nabla u_k\to \nabla u \quad \text{strongly in } L^{p}(B_{r_2}(x_0)).$$
Thus there exists a subsequence $\{k_j\}$ such that $\nabla u_{k_j}\to\nabla u$ almost everywhere in $B_{r_2}(x_0)$.

Since $A_k$ and $\mu_k$ are smooth in $x$, by the classical regularity theory (see, for instance, \cite{dibenedetto1983local},\cite{tolksdorf1984regularity}), we know that $u_k\in C^{1,\alpha}_{\text{loc}}(B_{r_2}(x_0))$. Therefore the Lipschitz estimate \eqref{infty-norm2} from Step 1 holds for $u_k$ in $B_{r_1/2}(x_0)$. Namely,
$$
\|\nabla u_k\|_{L^\infty(B_{r_1/2}(x_0))}
         \leq C\big\|\mathbf{I}_1^{r_1}(|\mu_k|)\big\|^\frac{1}{p-1}_{L^\infty(B_{r_1}(x_0))}+C r_1^{\frac{n}{2-p}} \||\nabla u_k|+s\|_{L^{2-p}(B_{r_1}(x_0))}.
$$
Note that by direct computation, for any $t>0$, it follows that
$$
\mu_k(B_t(x))=\int_{\R^n}\mu(B_t(x-y))\varphi_k(y)dy.
$$
Therefore, for sufficiently large $k$, by the Fubini--Tonelli theorem we have
$$
\big\|\mathbf{I}_1^{r_1}(|\mu_k|)\big\|_{L^\infty(B_{r_1}(x_0))}\leq \big\|\mathbf{I}_1^{r_1}(|\mu|)\big\|_{L^\infty(B_{r_2}(x_0))}.
$$
Thus by taking $k=k_j\nearrow\infty$ and then $r_1\nearrow R$, we obtain the Lipschitz estimate \eqref{ineq:int:lip} around $x=x_0$.
\end{proof}

In the rest of the section, we derive an interior modulus of continuity estimate of $\nabla u$ under the conditions of Theorem \ref{thm:int:lip}. Recall that we fixed an $\varepsilon\in (0,1/4)$ sufficiently small such that
$$
C\varepsilon^{\alpha-{\alpha_1}}<1\quad \text{and}\quad
\varepsilon^{\alpha_1}<1/4,
$$
where $C$ is the constant in \eqref{ineq:prop}, $\alpha\in(0,1)$ is the same constant as in Theorem \ref{thm:osc} and ${\alpha_1}\in (0,\alpha)$. We also took a ball $B_{R}(x_0)\subset \Omega$ with $R\in (0,1]$. Similar to  \eqref{eq6.42}, we define
\begin{equation}\label{def:tilde}
\begin{aligned}
 &\Tilde{\omega}(t)=\sum_{i=1}^\infty \varepsilon^{{\alpha_1} i}\left(\omega(\varepsilon^{-i}t)[\epsilon^{-i}t\leq R/2]+\omega(R/2)[\varepsilon^{-i}t>R/2]\right),\\
 &\Tilde{\mathbf{I}}_1^\rho(|\mu|)(x)=\sum_{i=1}^\infty \varepsilon^{{\alpha_1} i}\left(\mathbf{I}_1^{\varepsilon^{-i}\rho}(|\mu|)(x)[\varepsilon^{-i}\rho\leq R/2]+\mathbf{I}_1^{R/2}(|\mu|)(x)[\varepsilon^{-i}\rho>R/2]\right),\\
&\Tilde{\mathbf{W}}_{1/p,p}^\rho(|\mu|)(x)\\
&=\sum_{i=1}^\infty \varepsilon^{{\alpha_1} i}\left(\mathbf{W}_{1/p,p}^{\varepsilon^{-i}\rho}(|\mu|)(x)[\varepsilon^{-i}\rho\leq R/2]+\mathbf{W}_{1/p,p}^{R/2}(|\mu|)(x)[\varepsilon^{-i}\rho>R/2]\right).
\end{aligned}
\end{equation}
Here we also used the Iverson bracket notation, i.e., $[P] = 1$ if $P$ is true and $[P] = 0$ otherwise, and $\mathbf{I}_1$ and $\mathbf{W}_{1/p,p}$ are the Riesz and Wolff potentials defined in \eqref{def:riesz} and \eqref{def:wolff}, respectively. We note that since $1/(p-1)>1$, we have
\begin{equation}
                                \label{eq12.00}
\mathbf{W}_{1/p,p}^{\rho}(|\mu|)(x)\le C \big(\mathbf{I}_1^{2\rho}(|\mu|)(x)\big)^{\frac{1}{p-1}},
\end{equation}
so that
$\Tilde{\mathbf{W}}_{1/p,p}^\rho(|\mu|)(x)$ and $\Tilde{\mathbf{I}}_1^\rho(|\mu|)(x)$ are bounded and converge to zero as $\rho\to 0$ as long as $\mathbf{I}_1^{R}(|\mu|)(x)$ is finite.

Our interior modulus of continuity estimate is stated as follows.
\begin{theorem}
                    \label{thm4.1}
Assume the conditions of Theorem \ref{thm:int:lip} and ${\alpha_1}\in (0,\alpha)$, where $\alpha$ is the constant in Theorem \ref{thm:osc}. Then there exist a constant 
 $C=C(n,p,\lambda,\alpha_1,\omega)$, such that for any $R\in (0,1]$, $B_R(x_0)\subset \Omega,\; x,y\in B_{R/4}(x_0)$ being Lebesgue points of the vector-valued function $\nabla u$, it holds that

\begin{equation}\label{eq7.10}
\begin{aligned}
    &|\nabla u(x)-\nabla u(y)|\\
    &\leq C\,\mathbf{M}\Big[\left(\frac{\rho}{R}\right)^{\alpha_1}+\int_0^{\rho}\frac{\Tilde{\omega}(t)}{t}\,dt\Big]
    +C\,\big\|\Tilde{\mathbf{W}}_{1/p,p}^{\rho}(|\mu|)\big\|_{L^\infty(B_{R/4}(x_0))}\\
    &\quad +C\,\mathbf{M}^{2-p}
    \big\|\Tilde{\mathbf{I}}_1^{\rho}(|\mu|)\big\|_{L^\infty(B_{R/4}(x_0))}
\end{aligned}
\end{equation}
where $\rho=|x-y|$, $\Tilde{\omega}$, $\Tilde{\mathbf{W}}_{1/p,p}$, and $\Tilde{\mathbf{I}}_1$ are defined in \eqref{def:tilde}, and
$$
\mathbf{M}:=R^{-\frac{n}{2-p}}\||\nabla u|+s\|_{L^{{2-p}}(B_{R}(x_0))}+\big\|\mathbf{I}_1^{R}(|\mu|)\big\|^\frac{1}{p-1}_{L^\infty(B_{R}(x_0))}.
$$
\end{theorem}
Note that due to \eqref{eq12.00}, the term $\big\|\Tilde{\mathbf{W}}_{1/p,p}^{\rho}(|\mu|)\big\|_{L^\infty(B_{R/4}(x_0))}$ in \eqref{eq7.10} can be replaced with the sup norm of a summation of the truncated Riesz potentials similar to \eqref{def:tilde}.

\begin{proof}[Proof of Theorem \ref{thm4.1}]
For any $x,y\in B_{R/4}(x_0)$  being Lebesgue points of $\nabla u$, by the triangle inequality, we have
\begin{align*}
    &|\nabla u(x)-\nabla u(y)|^{\gamma_0}\\
    &\leq |\nabla u(x)-\mathbf{q}_{x,\rho}|^{\gamma_0}+|\mathbf{q}_{x,\rho}-\mathbf{q}_{y,\rho}|^{\gamma_0}+|\nabla u(y)-\mathbf{q}_{y,\rho}|^{\gamma_0}\\
    &\leq 2\sup_{y_0\in B_{R/4}(x_0)}|\nabla u(y_0)-\mathbf{q}_{y_0,\rho}|^{\gamma_0}+|\nabla u(z)-\mathbf{q}_{x,\rho}|^{\gamma_0}+|\nabla u(z)-\mathbf{q}_{y,\rho}|^{\gamma_0}.
\end{align*}
We set $\rho=|x-y|$, take the average over $z\in B(x,\rho)\cap B(y,\rho)$, and then take the $\gamma_0$-th root to get
\begin{align}
    |\nabla u(x)-\nabla u(y)|&\leq C \sup_{y_0\in B_{R/4}(x_0)}|\nabla u(y_0)-\mathbf{q}_{y_0,\rho}|+C\phi(x,\rho)+C\phi(y,\rho)\notag\\
    &\leq C \sup_{y_0\in B_{R/4}(x_0)}\sum_{j=0}^\infty \phi(y_0,\varepsilon^j\rho)+C \sup_{y_0\in B_{R/4}(x_0)}\phi(y_0,\rho)\notag\\
                \label{eq7.56}
    &\leq  C \sup_{y_0\in B_{R/4}(x_0)}\sum_{j=0}^\infty \phi(y_0,\varepsilon^j\rho).
\end{align}
Here we used (\ref{ineq:diff}) in the second inequality.

If $\rho< R/8$, by using \eqref{eq7.56}, (\ref{sum1:phi}) with $R/8$ in place of $r$, and the fact that
\begin{align*}
    B_{R/4}(y_0)\subset B_{R/2}(x_0)\quad \forall y_0\in B_{R/4}(x_0),
\end{align*}
we obtain
\begin{equation}\label{eq:diff}
\begin{aligned}
     &|\nabla u(x)-\nabla u(y)|\\
     &\leq  C \left(\frac{\rho}{R}\right)^{\alpha_1} \|\nabla u\|_{L^\infty(B_{R/2}(x_0))}+C \sup_{y_0\in B_{R/4}(x_0)}\int_0^{\rho}\frac{\Tilde{h}(y_0,t)}{t}\,dt
    \\&\quad +C\left(\|\nabla u\|_{L^\infty(B_{R/2}(x_0))}+s\right)^{2-p}\sup_{y_0\in B_{R/4}(x_0)}\int_0^{\rho}\frac{\Tilde{g}(y_0,t)}{t}\,dt\\&
    \quad +C\left(\|\nabla u\|_{L^\infty(B_{R/2}(x_0))}+s\right)\int_0^{\rho}\frac{\Tilde{\omega}(t)}{t}\,dt.
\end{aligned}
\end{equation}
Clearly, (\ref{eq:diff}) still holds when $\rho\geq R/8$.

We can simplify the terms in \eqref{eq:diff} as follows.
For any $y_0\in B_{R/4}(x_0)$ and $\rho\in (0,R/2)$, by the definition of $\Tilde{g}$ in \eqref{eq6.42}, we have
\begin{align*}
    \int_0^\rho\frac{\Tilde{g}(y_0,t)}{t}\,dt&=\sum_{i=1}^\infty \varepsilon^{{\alpha_1} i}\int_0^\rho\frac{g(y_0,\epsilon^{-i}t)}{t}[\epsilon^{-i}t\leq R/2]\,dt\\
    &\quad +\sum_{i=1}^\infty \varepsilon^{{\alpha_1} i}\int_0^\rho\frac{g(y_0,R/2)}{t}[\epsilon^{-i}t>R/2]\,dt.
\end{align*}
Recalling the definition of $\Tilde{\mathbf{I}}_1$ from \eqref{def:tilde}, the first term above is equal to
\begin{align*}
   &\sum_{i=1}^\infty\varepsilon^{{\alpha_1} i}\Big(\int_0^\rho\frac{g(y_0,\varepsilon^{-i}t)}{t}\,dt[\varepsilon^{-i}\rho\leq R/2]+\int_0^{\epsilon^i R/2}\frac{g(y_0,\varepsilon^{-i}t)}{t}\,dt[\varepsilon^{-i}\rho>R/2]\Big)\\
    &=\sum_{i=1}^\infty \varepsilon^{{\alpha_1} i}\Big(\int_0^{\varepsilon^{-i}\rho}\frac{g(y_0,t)}{t}\,dt[\varepsilon^{-i}\rho\leq R/2]+\int_0^{R/2}\frac{g(y_0,t)}{t}\,dt[\varepsilon^{-i}\rho>R/2]\Big)\\
    &=\sum_{i=1}^\infty \varepsilon^{{\alpha_1} i}\Big(\mathbf{I}_1^{\varepsilon^{-i}\rho}(|\mu|)(y_0)[\varepsilon^{-i}\rho\leq R/2]+\mathbf{I}_1^{R/2}(|\mu|)(y_0)[\varepsilon^{-i}\rho>R/2]\Big)\\
    &=\Tilde{\mathbf{I}}_1^\rho(|\mu|)(y_0).
\end{align*}
The second term is  equal to
\begin{align*}
  \sum_{i=1}^\infty\varepsilon^{{\alpha_1} i}[\epsilon^{-i}\rho>R/2]\ln(2\varepsilon^{-i}\rho/R)g(y_0,R/2).
\end{align*}
Let $K$ be the positive integer such that
$ \epsilon^{-K}\rho>R/2$ and $\epsilon^{-(K-1)}\rho\leq R/2$. Then we have
\begin{align*}
    &\sum_{i=1}^\infty\varepsilon^{{\alpha_1} i}[\epsilon^{-i}\rho>R/2]\ln(2\varepsilon^{-i}\rho/R)
    =\sum_{i=K}^\infty\varepsilon^{{\alpha_1} i}\ln(2\varepsilon^{-i}\rho/R)\\
    &=\epsilon^{\alpha_1K}\sum_{i=K}^\infty\varepsilon^{{\alpha_1} (i-K)}\Big(\ln(2\varepsilon^{-K}\rho/R)+(i-K)\ln(\epsilon^{-1})\Big)\\
    &\leq  \Big(\frac{2\rho}{R}\Big)^{\alpha_1}\sum_{i=K}^\infty\varepsilon^{{\alpha_1} (i-K)}(i-K+1)\ln(\epsilon^{-1})\leq C \Big(\frac{\rho}{R}\Big)^{\alpha_1}.
\end{align*}
Therefore,
\begin{equation}\label{g}
\begin{aligned}
     \int_0^{\rho}\frac{\Tilde{g}(y_0,t)}{t}\,dt\leq& \Tilde{\mathbf{I}}_1^{\rho}(|\mu|)(y_0)+C(\rho/R)^{\alpha_1}g(y_0,R/2)\\
     \leq&\Tilde{\mathbf{I}}_1^{\rho}(|\mu|)(y_0)+C(\rho/R)^{\alpha_1}\mathbf{I}_1^R(|\mu|)(y_0).
\end{aligned}
\end{equation}

We can similarly get the following estimate
\begin{equation}
\begin{aligned}\label{h}
    &\int_0^\rho\frac{\Tilde{h}(y_0,t)}{t}\,dt\leq  \sum_{i=1}^\infty
    \left(\varepsilon^{{\alpha_1} i}\int_0^{\varepsilon^{-i}\rho}h(y_0,t)\,\frac{dt}{t}[\varepsilon^{-i}\rho\leq R/2]\right.\\
    &\qquad \left. +\int_0^{R/2}h(y_0,t)\,\frac{dt}{t}[\varepsilon^{-i}\rho>R/2]\right)+C\left(\frac{\rho}{R}\right)^{\alpha_1} h(y_0,R/2)\\
    &\qquad\qquad\qquad\leq\Tilde{\mathbf{W}}_{1/p,p}^{\rho}(|\mu|)(y_0)+C(\rho/R)^{\alpha_1}\big(\mathbf{I}_1^R(|\mu|)(y_0)\big)^{1/(p-1)}.
\end{aligned}
\end{equation}
Using (\ref{ineq:int:lip}), (\ref{g}), and (\ref{h}), from (\ref{eq:diff}) we obtain \eqref{eq7.10}. The theorem is proved.
\end{proof}

\begin{proof}[Proof of Theorem \ref{thm1.2}]
Since the set of Lebesgue points of $\nabla u$ is dense in $\Omega$, it suffices to show the right-hand side of \eqref{eq7.10} converges to zero when $\rho\to 0$. In fact, we have
\begin{align*}
    &\|\Tilde{\mathbf{I}}_1^{\rho}(|\mu|)\|_{L^\infty(B_{R/4}(x_0))}
\leq\sum_{i=1}^\infty \varepsilon^{{\alpha_1} i}\left( \|\mathbf{I}_1^{\varepsilon^{-i}\rho}(|\mu|)\|_{L^\infty(B_{R/4}(x_0))}[\varepsilon^{-i}\rho\leq R/2]\right.\\
&\qquad\left.+\|\mathbf{I}_1^{R/2}(|\mu|)\|_{L^\infty(B_{R/4}(x_0))}[\varepsilon^{-i}\rho>R/2]\right),
\end{align*}
which must converge to $0$ by using \eqref{asp1} and the dominated convergence theorem. We can similarly prove the convergence of other terms in \eqref{eq7.10}.
\end{proof}

\begin{proof}[Proof of Corollary \ref{thm1.3}]
By \cite[Lemma 3]{MR2729305} with $p=2$ and $k=1$, the assumption \eqref{asp2} implies \eqref{asp1}. Therefore Corollary \ref{thm1.3} follows from Theorem \ref{thm1.2}.
\end{proof}

\begin{proof}[Proof of Corollary \ref{thm1.4}]
This is an immediate corollary of Theorem \ref{thm1.2} since the assumption \eqref{asp1} is verified by \eqref{asp3} and \eqref{asp4}.
\end{proof}

\begin{proof}[Proof of Corollary \ref{thm1.8}]
We choose $\alpha_1\in (\beta,\alpha)$ in Theorem \ref{thm4.1}. Then Corollary \ref{thm1.8} follows by a direct computation using \eqref{eq7.10}.
\end{proof}


\section{Global gradient estimates for the \texorpdfstring{$p$}{p}-Laplacian equations}\label{sec5}
This section is devoted to the proof of the global pointwise gradient estimate in Theorem \ref{thm:bdry}, the Lipschitz estimate in Theorem \ref{thm:bdry:lip}, Corollary \ref{thm:glo:lip}, as well as the derivation of a global modulus of continuity estimate of $\nabla u$ stated in Theorem \ref{thm:bdry2} for the following (possibly nondegenerate) $p$-Laplace equation with Dirichlet boundary condition:
\begin{equation}\label{eq:p}
    \left\{
\begin{aligned}
      -\div\left(a(x)(|\nabla u|^2+s^2)^\frac{p-2}{2}\nabla u\right)&=\mu \quad \text{in} \,\, \Omega,  &\\
     u &=  0 \quad \text{on}\,\, \partial \Omega, &\\
\end{aligned}
\right.
\end{equation}
where  $a(\cdot)$ satisfies \eqref{dini}, \eqref{bdd}, and \eqref{dini-cts}, and $\Omega$ has a $C^{1,\text{Dini}}$ boundary characterized by $R_0$ and $\omega_0$ as in Definition \ref{def:bdry}.

First, we derive a gradient estimate around any point $x_0\in \partial\Omega$. Without loss of generality, we assume that $x_0=0\in \partial \Omega$. Then we can choose a local coordinate around $x_0=0$ and a function $\chi$ as in Definition \ref{def:bdry} such that $\chi(0')=0$. Let
$$
\Gamma(y)=(y_1+\chi(y'),y')\quad \text{and}\quad
\Lambda(x)= \Gamma^{-1}(x)=(x_1-\chi(x'),x').
$$
Note that the determinants of the Jacobian of $\Gamma(\cdot)$ and $\Lambda(\cdot)$ are equal to $1$.
Since $\Omega$ has $C^{1,\text{Dini}}$ boundary, from the proof of \cite[Lemma 2.2]{choi2019gradient}, there exists $R_1=R_1(\omega_0,R_0)\in(0,R_0)$ such that
\begin{equation}\label{chi}
   |\nabla_{x'}\chi(x')|\leq 1/2 \quad \text{if} \quad |x'|\leq R_1,
\end{equation}
\begin{equation}\label{bdry}
    \Omega_{r/2}\subset \Gamma(B^+_r)\subset \Omega_{2r} \quad \forall\, r\in(0,R_1/2].
\end{equation}
Therefore, there exist constants $c_1(n)$ and $c_2(n)$ depending only on $n$, such that for any $x\in \Bar{\Omega}$ and $0<r\leq R_1$,
\begin{equation}\label{volume}
    c_1(n) r^n\leq  |\Omega_r(x)| \leq c_2(n)r^n.
\end{equation}

Now we use the technique of flattening the boundary. We denote
$u_1(y)=u(\Gamma(y))$, $a_1(y)=a(\Gamma(y))$, and $\mu_1(A)=\mu(\Gamma(A))$
for any Borel set $A\subset \R^n$. Then $u_1$ satisfies
\begin{equation}\label{eq:u1}
    \left\{
\begin{aligned}
     \div_y\left(a_1(y)(|D\Lambda\; D_y u_1|^{2}+s^2)^\frac{p-2}{2}(D\Lambda)^T D\Lambda \;D_y u_1\right)=\;&\mu_1 \quad  \text{in} \,\, B^+_{{R_1}/{2}},\\
     u_1=\;&0 \quad \text{on}\,\, B_{{R_1}/{2}}\cap\partial \mathbb{R}^n_+.
\end{aligned}
\right.
\end{equation}
We set
\begin{align*}
    A_1(y,\xi)=a_1(y)(|D\Lambda\; \xi|^{2}+s^2)^\frac{p-2}{2}(D\Lambda)^T D\Lambda \;\xi.
\end{align*}
By direct computations with \eqref{chi} in hand, $A_1$ satisfies the following conditions with $\omega_1=\omega+\omega_0$ and some constant $\lambda_1=\lambda_1(n,p,\lambda)$:
\begin{equation}\label{growth1}
    |A_1(y,\xi)|\leq\lambda_1(s^2+|\xi|^2)^{(p-1)/2}, \quad |D_\xi A_1(y,\xi)|\leq\lambda_1(s^2+|\xi|^2)^{(p-2)/2},
\end{equation}
\begin{equation}\label{elliptic1}
    \left\langle D_\xi A_1(y,\xi)\eta,\eta \right\rangle\geq\lambda_1^{-1}(s^2+|\xi|^2)^{(p-2)/2}|\eta|^2,
\end{equation}
\begin{equation}\label{osc:bdry1}
    |A_1(y,\xi)-A_1(y_0,\xi)|\leq \lambda_1 \omega_1(|y-y_0|)(s^2+|\xi|^2)^{(p-1)/2}
\end{equation}
for every $y,y_0\in B^{+}_{R_1/2}$ and every $(\xi,\eta)\in\mathbb{R}^n\times\mathbb{R}^n\backslash\{(0,0)\}$.

Suppose that $4r\leq R_1$. We now consider the unique solution $w\in u_1+W_0^{1,p}(B^+_{2r})$ to the equation
\begin{equation}\label{eq:w1}
    \left\{
\begin{aligned}
     -\div_y(A_1(y,\nabla_y w)) =&  0 \quad \text{in} \,\, B^+_{2r},  \\
     w =&  u_1 \quad \text{on}\,\, \partial B^+_{2r}. \\
\end{aligned}
\right.
\end{equation}

We first derive a boundary version of the reverse H\"older's inequality.
\begin{lemma}
                    \label{lem5.1}
Let w be a solution to \eqref{eq:w1}. There exists a constant ${\theta_1}>p$ depending only on $n$, $p$, and $\lambda$, such that for any $t>0$, the estimate
\begin{equation}\label{reverse}
    \left (\fint_{B^+_{\rho/2}(y_0)}(|\nabla_y w|+s)^{{\theta_1}}\,dx\right)^{1/{\theta_1}}\leq C\left(\fint_{B^+_{\rho}(y_0)}(|\nabla_y w|+s)^t\,dx\right)^{1/t}
\end{equation}
holds for all $B^+_\rho(y_0)\subset B^+_{2r}$, where $C=C(n,p,\lambda,t)>0$.
\end{lemma}

\begin{proof}
For simplicity, we still denote $\nabla=\nabla_y$ through this proof.
First we prove a Caccioppoli type inequality in half balls. Suppose that $y_0\in B_{2r}\cap \partial\R^n_+$ and $B_{2\rho}(y_0)\subset\subset B_{2r}$. Let $\zeta$ be a nonnegative smooth function satisfying $\zeta=1$ in $B_\rho(y_0) $, $|\nabla \zeta|\leq 2\rho^{-1}$, and $\zeta=0$ outside $B_{2\rho}(y_0)$.
Using $\zeta^p w$ as a test function in (\ref{eq:w1}), we get the following
\begin{equation}\label{test}
0=\int_{B^+_{2r}} \langle A_1(y, \nabla w),\zeta^p \nabla w\rangle \,dy+p \int_{B^+_{2r}} \langle A_1(y, \nabla w), \zeta^{p-1} w \nabla \zeta \rangle\,dy:=\RN{1}+\RN{2}.
\end{equation}
Using the fundamental theorem of calculus, (\ref{growth1}), (\ref{elliptic1}), and Young's inequality with exponents $p$ and $p/(p-1)$, we have
\begin{align*}
   \RN{1}&=\int_{B^+_{2r}} \zeta^p \Big(\int_0^1 \langle D_\xi A_1(y,t\nabla w)\nabla w, \nabla w\rangle\; dt+A_1(y,0)\nabla w\Big)dy\\
    &\geq \int_{B^+_{2r}} \zeta^p \Big(\int_0^1 \lambda_1^{-1}(s^2+|\nabla w|^2t^2)^\frac{p-2}{2} |\nabla w|^2\,dt-\lambda_1 s^{p-1}|\nabla w|\Big)dy\\
    &\geq c(\lambda_1, p) \int_{B^+_{2r}} \zeta^p |\nabla w|^p \,dy-C'(\lambda_1,p) \int_{B^+_{2r}} \zeta^p s^p \,dy .
\end{align*}
On the other hand, using \eqref{growth1} and Young's inequality, we have
\begin{align*}
     &|\RN{2}|\leq p\lambda_1 \int_{B^+_{2\rho}(y_0)} (s^2+|\nabla w|^2)^\frac{p-1}{2} \zeta^{p-1} |w \nabla \zeta| \,dy\\
     &\leq \frac{1}{2} c(\lambda_1, p)\int_{B^+_{2r}}\zeta^{p} (s^2+|\nabla w|^2)^\frac{p}{2}\,dy+ C''(\lambda_1, p)\int_{B^+_{2\rho}(y_0)} |w \nabla \zeta|^p \,dy\\
     &\leq \frac{1}{2} c(\lambda_1, p)\int_{B^+_{2r}}\zeta^{p} (s^p+|\nabla w|^p)\,dy+ C''(\lambda_1, p)\int_{B^+_{2\rho}(y_0)} |w \nabla \zeta|^p \,dy.
\end{align*}
Therefore, (\ref{test}) implies the following Caccioppoli type inequality
\begin{equation}\label{caccio}
    \int_{B^+_\rho(y_0)}|\nabla w|^p\,dx\leq C \rho^n s^p+ C\rho^{-p}\int_{B^+_{2\rho}(y_0)}|w|^p\,dx.
\end{equation}
Since $w=u=0$ on $B_{2r}\cap \partial \R^n_+$, by the Sobolev-Poincar\'e inequality,
\begin{equation}\label{sob}
    \left(\int_{B^+_{2\rho}(y_0)}|w|^p\,dx\right)^{1/p}
    \leq C \rho^{1+\frac{n}{p}-\frac{n}{q}}   \left(\int_{B^+_{2\rho}(y_0)}|\nabla w|^q\,dx\right)^{1/q}
\end{equation}
for any $q$ such that $\max\big\{1,\frac{np}{n+p}\big\}\le q<p$.
Thus by combining (\ref{caccio}) and (\ref{sob}), we have
\begin{align*}
     \left (\fint_{B^+_{\rho}(y_0)}(|\nabla w|+s)^{p}\,dx\right)^{1/p}\leq C\left(\fint_{B^+_{2\rho}(y_0)}(|\nabla w|+s)^q \,dx\right)^{1/q}.
\end{align*}
Similarly, the following interior version of estimate
\begin{align*}
     \left (\fint_{B_{\rho}(y_0)}(|\nabla w|+s)^{p}\,dx\right)^{1/p}\leq C\left(\fint_{B_{2\rho}(y_0)}(|\nabla w|+s)^q \,dx\right)^{1/q}
\end{align*}
holds for all $B_{2\rho}(y_0)\subset\subset B^+_{2r}$.
Therefore, by a standard covering argument and Gehring's lemma, we get (\ref{reverse}).
\end{proof}

We also have a boundary comparison result analogous to Lemma \ref{lem:u-w} by following almost the same proof.
\begin{lemma}\label{lem:u-w2}
Let w be a solution to \eqref{eq:w1}  and assume that $p\in (1,2)$. Then for any $\gamma_0\in(0,2-p]$ when $p\in \big(\frac{3n-2}{2n-1},2\big)$ and $\gamma_0\in \big(0,\frac{(p-1)n}{n-1}\big)$ when $p\in \big(1, \frac{3n-2}{2n-1}\big]$, it holds that
\begin{align*}
         &\left(\fint_{B^+_{2r}}|\nabla_y u_1-\nabla_y w|^{\gamma_0}\,dx \right)^{1/\gamma_0}\\
         &\leq C\left[\frac{|\mu_1|(B^+_{2r})}{r^{n-1}} \right]^\frac{1}{p-1}+C\frac{|\mu_1|(B^+_{2r})}{r^{n-1}}\fint_{B^+_{2r}}(|\nabla_y u_1|+s)^{2-p}\,dx,
\end{align*}
where $C$ is a constant depending only on $n$, $p$, $\lambda$, and $\gamma_0$.
\end{lemma}

We now let $v\in w+W_0^{1,p}(B^+_r)$ be the unique solution to

\begin{equation}\label{eq:v2}
    \left\{
\begin{aligned}
     -\div_y(A_1(0,\nabla_y v)) =&  0 \quad \text{in} \,\, B^+_{r},  &\\
     v =&  w \quad \text{on}\,\, \partial B^+_{r}. &\\
\end{aligned}
\right.
\end{equation}
We also have an estimate for the difference $\nabla v-\nabla w$ analogous to \eqref{comp} by following almost the same proof as that of  \cite[Eq. (4.35)]{duzaar2010gradient}:
\begin{align*}
    \fint_{B^+_r}|\nabla_y v-\nabla_y w|^p\,dx\leq C\omega_0(r)^p\fint_{B^+_{r}}(|\nabla_y w|+s)^p\,dx.
\end{align*}
Thus by \eqref{reverse} and H\"older's inequality, we get
\begin{equation*}
        \fint_{B^+_r}|\nabla_y v-\nabla_y w|^{\gamma_0}\,dx\leq C\omega_0(r)^{\gamma_0}\fint_{B^+_{2r}}(|\nabla_y w|+s)^{\gamma_0}\,dx.
\end{equation*}
Next we prove an oscillation estimate for $v$.
\begin{lemma}
                \label{lem5.3}
Let $v$ be a solution to \eqref{eq:v2}. There exists constant $C>1$  depending only on $n$, $p$, $\lambda$, and $\gamma_0$, such that for any half ball $B_\rho^+\subset B_R^+\subset B_r^+$, we have
\begin{equation}\label{ineq:osc2}
\begin{aligned}
&\inf_{\theta\in \mathbb{R}}\left(\fint_{B^+_{ \rho}}|D_{y_1} v-\theta|^{\gamma_0}+|D_{y'}v|^{\gamma_0}\right)^{1/\gamma_0}\\
&\leq C \left(\frac{\rho}{R}\right)^\alpha \inf_{\theta\in \mathbb{R}}\left(\fint_{B^+_{R}}|D_{y_1} v-\theta|^{\gamma_0}+|D_{y'}v|^{\gamma_0}\right)^{1/\gamma_0},
\end{aligned}
\end{equation}
where $\alpha\in(0,1)$ is the same constant as in Theorem \ref{thm:osc}.
\end{lemma}
\begin{proof}
Let $\Bar{v}$ be an odd extension of $v$ in $B_{r}$,
namely,
\begin{equation*}
   \Bar{v}(y)= \left\{
\begin{aligned}
    &\qquad\quad  v(y)\quad\ \   \text{if} \quad y_1\geq 0, \\
    &-v(-y_1,y') \quad \text{if} \quad y_1<0.
\end{aligned}
\right.
\end{equation*}
Then since $$ A_1(0,\xi)=a_1(0)(|\xi|^{2}+s^2)^\frac{p-2}{2}\xi,$$
$\Bar{v}\in W^{1,p}(B_r)$ is a solution to the equation
$$
-\div_y\left(a_1(0)(|\nabla_y \Bar{v}|^2+s^2)^\frac{p-2}{2}\nabla_y \Bar{v}\right)=0 \quad \text{in} \,\, B_r.
$$
Thus we can apply Theorem \ref{thm:osc} to $\Bar{v}$ to get
\begin{align*}
     \inf_{\mathbf{q}\in \mathbb{R}^n}\left(\fint_{B_{\rho}}|\nabla_y \Bar{v}-\mathbf{q}|^{\gamma_0}\right)^{1/\gamma_0}\leq& C\left(\frac{\rho}{R}\right)^{\alpha}\inf_{\mathbf{q}\in \mathbb{R}^n}\left(\fint_{B_{R}}|\nabla_y \Bar{v}-\mathbf{q}|^{\gamma_0}\right)^{1/\gamma_0}.
\end{align*}
Since $\Bar{v}$ is an odd function in $y_1$, by the triangle inequality, there exists $\theta_\rho \in \R$ such that
\begin{align*}
    \left(\fint_{B^+_{ \rho}}|D_{y_1} v-\theta_\rho|^{\gamma_0}+|D_{y'}v|^{\gamma_0}\right)^{1/\gamma_0}\leq C  \inf_{\mathbf{q}\in \mathbb{R}^n}\left(\fint_{B_{\rho}}|\nabla_y \Bar{v}-\mathbf{q}|^{\gamma_0}\right)^{1/\gamma_0}.
\end{align*}
By the triangle inequality again, it is easily seen that
\begin{align*}
    \inf_{\mathbf{q}\in \mathbb{R}^n}\left(\fint_{B_{R}}|\nabla_y \Bar{v}-\mathbf{q}|^{\gamma_0}\right)^{1/\gamma_0}\leq C
    \inf_{\theta\in \mathbb{R}}\left(\fint_{B^+_{R}}|D_{y_1} v-\theta|^{\gamma_0}+|D_{y'}v|^{\gamma_0}\right)^{1/\gamma_0}.
\end{align*}
Then (\ref{ineq:osc2}) is a direct consequence of the three inequalities above.
\end{proof}

\begin{lemma}
Suppose that $u_1\in W^{1,p}(B^+_{R_1})$ is a solution to  (\ref{eq:u1}). Then for any $\varepsilon\in(0,1)$ and $r\in (0,R_1/4]$, we have
\begin{equation}\label{ineq:pr}
\begin{aligned}
&\inf_{\theta\in \mathbb{R}}\left(\fint_{B^+_{ \varepsilon r}}|D_{y_1} u_1-\theta|^{\gamma_0}+|D_{y'}u_1|^{\gamma_0}\right)^{1/\gamma_0}\\
&\leq C\varepsilon^{\alpha}\inf_{\theta\in \mathbb{R}}\left(\fint_{B^+_{ r}}|D_{y_1} u_1-\theta|^{\gamma_0}+|D_{y'}u_1|^{\gamma_0}\right)^{1/\gamma_0}\\
&\quad +C_\varepsilon\left(\frac{|\mu_1|(B^+_{2r})}{r^{n-1}}\right)^\frac{1}{p-1}+C_\varepsilon\omega_1(r)\left(\fint_{B^+_{2r}}(|\nabla_y u_1|+s)^{2-p}\right)^\frac{1}{2-p}\\
&\quad +C_\varepsilon \frac{|\mu_1|(B^+_{2r})}{r^{n-1}}\fint_{B^+_{2r}}(|\nabla_y u_1|+s)^{2-p},
\end{aligned}
\end{equation}
where  $\alpha$ and $\gamma_0$ are the same constants as in Proposition \ref{prop1}, $C_\varepsilon$ is a constant depending on $\varepsilon$, $n$, $p$, $\lambda$, and $\gamma_0$, and $C$ is a constant depending on $n$, $p$, $\lambda$, and $\gamma_0$.
\end{lemma}
\begin{proof}
By using Lemmas \ref{lem5.1}, \ref{lem:u-w2}, and \ref{lem5.3},
the proof is almost identical to that of Proposition \ref{prop1}, so we omit it.
\end{proof}
We now define
$$\psi(x,r)=\inf_{\theta\in \mathbb{R}}\left(\fint_{\Omega_{r}(x)}|D_1 u-\theta|^{\gamma_0}+|D_{x'}u|^{\gamma_0}\right)^{1/\gamma_0}.$$
Let $\varepsilon\in(0,1)$ and  $r\in (0, R_1/4]$. By using change of variables, \eqref{bdry}, \eqref{volume}, and the triangle inequality,  we have
\begin{equation}\label{eq:geq}
\begin{aligned}
    &\inf_{\theta\in \mathbb{R}}\left(\fint_{B^+_{ \varepsilon r}}|D_{y_1} u_1-\theta|^{\gamma_0}+|D_{y'}u_1|^{\gamma_0}\right)^{1/\gamma_0}\\
    &= \inf_{\theta\in \mathbb{R}}\left(\fint_{\Gamma(B^+_{ \varepsilon r})}|D_{1} u-\theta|^{\gamma_0}+|D_{1}u\; D_{x'}\chi+D_{x'} u|^{\gamma_0}\right)^{1/\gamma_0}\\
    &\geq C \inf_{\theta\in \mathbb{R}}\left(\fint_{\Omega_{\epsilon r/2}}|D_{1} u-\theta|^{\gamma_0}+|D_{x'} u|^{\gamma_0}\right)^{1/\gamma_0}-C' \left(\fint_{\Omega_{\epsilon r/2}} |D_{1}u\; D_{x'}\chi|^{\gamma_0}\right)^{1/\gamma_0} \\
    &\geq C\psi(0,\epsilon r/2)-C'\omega_1(\epsilon r/2)\left(\fint_{\Omega_{\epsilon r/2}} |\nabla u|^{\gamma_0}\right)^{1/\gamma_0},
\end{aligned}
\end{equation}
where $C$ and $C'$ are positive constants depending only on $n$ and $\gamma_0$.
Similarly,
\begin{equation}\label{eq:leq}
\begin{aligned}
    &\inf_{\theta\in \mathbb{R}}\left(\fint_{B^+_{ r}}|D_{y_1} u_1-\theta|^{\gamma_0}+|D_{y'}u_1|^{\gamma_0}\right)^{1/\gamma_0}\\
     &= \inf_{\theta\in \mathbb{R}}\left(\fint_{\Gamma(B^+_{ r})}|D_{1} u-\theta|^{\gamma_0}+|D_{1}u\; D_{x'}\chi+D_{x'} u|^{\gamma_0}\right)^{1/\gamma_0}\\
    &\leq C'' \inf_{\theta\in \mathbb{R}}\left(\fint_{\Omega_{ 2r}}|D_{1} u-\theta|^{\gamma_0}+|D_{x'} u|^{\gamma_0}\right)^{1/\gamma_0}+C'' \left(\fint_{\Omega_{2r}} |D_{1}u\; D_{x'}\chi|^{\gamma_0}\right)^{1/\gamma_0} \\
    &\leq C''\psi(0,2r)+C''\omega_1(2r)\left(\fint_{\Omega_{2r}} |\nabla u|^{\gamma_0}\right)^{1/\gamma_0},
\end{aligned}
\end{equation}
where $C''$ is a positive constant depending only on $n$ and $\gamma_0$. Therefore, by using \eqref{eq:geq}, \eqref{eq:leq}, \eqref{chi}, and \eqref{volume}, \eqref{ineq:pr} implies that
\begin{align*}
        &\psi(0,\epsilon r/2)\leq C\varepsilon^{\alpha}\psi(0,2r)+C_\varepsilon\left(\frac{|\mu|(\Omega_{4r})}{r^{n-1}}\right)^\frac{1}{p-1}
        \\&\quad+C_\varepsilon\omega_1(2r)\left(\fint_{\Omega_{4r}}(|\nabla u|+s)^{2-p}\right)^\frac{1}{2-p}+C_\varepsilon\frac{|\mu|(\Omega_{4r})}{r^{n-1}}\fint_{\Omega_{4r}}(|\nabla u|+s)^{2-p}.
\end{align*}
By replacing $\varepsilon/4$ and $2r$ with $\varepsilon$ and $r$ respectively, we obtain
\begin{corollary}
Suppose that $u\in W^{1,p}_0(\Omega)$ is a solution to \eqref{eq:p} and $x_0\in\partial \Omega$. Then for $\varepsilon\in(0,1/4)$, $r\leq R_1/2$, and $\alpha$, $C$, $C_\epsilon$ as above, we have
\begin{equation}\label{ineq:cor}
\begin{aligned}
       &\psi(x_0,\epsilon r)\leq C\varepsilon^{\alpha}\psi(x_0, r)+C_\varepsilon\left(\frac{|\mu|(\Omega_{2r}(x_0))}{r^{n-1}}\right)^\frac{1}{p-1}\\
       &\,\,  +C_\varepsilon\omega_1(r)\left(\fint_{\Omega_{2r}(x_0)}(|\nabla u|+s)^{2-p}\right)^\frac{1}{2-p}+C_\varepsilon\frac{|\mu|(\Omega_{2r}(x_0))}{r^{n-1}}\fint_{\Omega_{2r}(x_0)}(|\nabla u|+s)^{2-p}.
\end{aligned}
\end{equation}
\end{corollary}
As in Section \ref{sec3}, we define
\begin{equation*}
    \phi(x,\rho)=\inf_{\mathbf{q}\in \R^n}\left(\fint_{\Omega_\rho(x)}|\nabla u-\mathbf{q}|^{\gamma_0}\right)^{1/\gamma_0}
\end{equation*}
and choose $\mathbf{q}_{x,r}\in \R^n$ such that
\begin{equation}\label{def:qxr}
    \left(\fint_{\Omega_r(x)}|\nabla u-\mathbf{q}_{x,r}|^{\gamma_0}\right)^{1/\gamma_0}=\inf_{\mathbf{q}\in \R^n}\left(\fint_{\Omega_r(x)}|\nabla u-\mathbf{q}|^{\gamma_0}\right)^{1/\gamma_0}.
\end{equation}
We remark that \eqref{limit} still holds for any Lebesgue point $x\in \Omega$ of the vector-valued function $\nabla u$ from the same argument as in Section \ref{sec3}. Moreover, if we assume $u\in C^1(\Bar{\Omega})$, then \eqref{limit} actually holds for any $x\in \Bar{\Omega}$.
\subsection{Global pointwise gradient estimates}

To prove the pointwise gradient estimate for  $p\in \big(\frac{3n-2}{2n-1},2\big)$, we choose
$\gamma_0=2-p$ and $\varepsilon=\varepsilon(n,p,\lambda,\alpha)\in (0,1/4)$ sufficiently small such that $C\epsilon^{\alpha}\leq 1/4$ for both constants $C$ in \eqref{ineq:prop} and \eqref{ineq:cor}. Fix $x_0\in \partial \Omega$ and $R\leq R_1/2$. For $j\ge 0$, set $r_j=\varepsilon^j R$, $\Omega_j=\Omega_{2r_j}(x_0)$,
$$
T_j=\left(\fint_{\Omega_j}(|\nabla u|+s)^{2-p}\,dx\right)^\frac{1}{2-p},\quad \phi_j=\phi(x_0,r_j),\quad\text{and}\,\,
\psi_j= \psi(x_0,r_j).
$$
Applying \eqref{ineq:cor} yields
\begin{equation*}
    \psi_{j+1}\leq\frac{1}{4}\psi_j+C\left(\frac{|\mu|(\Omega_j)}{r_j^{n-1}}\right)^\frac{1}{p-1}
    +C\frac{|\mu|(\Omega_j)}{r_j^{n-1}}T_j^{2-p}+C\omega_1(r_j) T_j.
\end{equation*}
Let $j_0$ and $m$ be positive integers such that $j_0\le m$. Summing the above inequality over $j\in\{j_0,j_0+1,\dots,m\}$ and noting that $\phi_j\leq \psi_j\leq C T_j$, we obtain that
\begin{equation}\label{sum2}
\begin{aligned}
    \sum_{j=j_0}^{m+1}\phi_j\leq\sum_{j=j_0}^{m+1}\psi_j&\leq CT_{j_0}+C\sum_{j=j_0}^{m}\left(\frac{|\mu|(\Omega_j)}{r_j^{n-1}}\right)^\frac{1}{p-1}\\
    &\quad +C\sum_{j=j_0}^m \frac{|\mu|(\Omega_j)}{r_j^{n-1}}T_j^{2-p}
    +C\sum_{j=j_0}^m\omega_1(r_j)T_j
\end{aligned}
\end{equation}
for any $x_0\in \partial \Omega$ and $R\leq R_1/2$.

On the other hand, according to \eqref{sum:int},
\begin{equation}\label{sum:int2}
\begin{aligned}
    \sum_{j=j_0}^{m+1}\phi_j&\leq C\phi_{j_0}+C\sum_{j=j_0}^{m}\left(\frac{|\mu|(\Omega_j)}{r_j^{n-1}}\right)^\frac{1}{p-1}\\
   &\quad  +C\sum_{j=j_0}^m \frac{|\mu|(\Omega_j)}{r_j^{n-1}}T_j^{2-p}+C\sum_{j=j_0}^m\omega(r_j)T_j
\end{aligned}
\end{equation}
holds for any $x_0\in \Omega$ and $R>0$ such that $r_{j_0}=\epsilon^{j_0}R <\text{dist}(x_0,\partial\Omega)/2$.

We now define
$$
\Omega_j^\prime=\Omega_{8r_j}(x_0),\quad
Z_j=\left(\fint_{\Omega_j^\prime}(|\nabla u|+s)^{2-p}\,dx\right)^\frac{1}{2-p}.
$$
Then we can obtain the following lemma.
\begin{lemma}
Suppose that $u\in W^{1,p}_0(\Omega)$ is a solution to \eqref{eq:p}, $x_0\in \Bar{\Omega}$, and $R\leq R_1/6$. Then we have
\begin{equation}\label{sum3}
\begin{aligned}
\sum_{j=j_0}^{m+1}\phi_j&\leq CZ_{j_0}+C\sum_{j=j_0}^{m}\left(\frac{|\mu|(\Omega_j^\prime)}{r_j^{n-1}}\right)^\frac{1}{p-1}\\
&\quad +C\sum_{j=j_0}^m \frac{|\mu|(\Omega_j^\prime)}{r_j^{n-1}} Z_j^{2-p}
+C\sum_{j=j_0}^{m+1}\omega_1(r_j)Z_j,
\end{aligned}
\end{equation}
where $C$ is a constant depending only on $n$, $p$, $\lambda$, and $\gamma_0$.
\end{lemma}
\begin{proof}
First when $x_0\in \partial\Omega$, since $\Omega_j\subset \Omega_j^\prime$, we have $T_j\leq C Z_j$. Thus \eqref{sum2} directly implies \eqref{sum3}. It remains to prove the lemma for $x_0\in \Omega$. Since \eqref{sum:int2} holds when $r_{j_0}<\text{dist}(x_0,\partial\Omega)/2$, we only need to show that \eqref{sum3} holds when $r_{j_0}\geq \text{dist}(x_0,\partial\Omega)/2$.
Now assume $r_{j_1}\geq \text{dist}(x_0,\partial\Omega)/2$ and $r_{j_1+1}< \text{dist}(x_0,\partial\Omega)/2$.
By \eqref{sum:int2}, we have
\begin{equation}\label{sum:part1}
\begin{aligned}
    \sum_{j=j_1+1}^{m+1}\phi_j&\leq C\phi_{j_1+1}+C\sum_{j=j_1+1}^{m}\left(\frac{|\mu|(\Omega_j)}{r_j^{n-1}}\right)^\frac{1}{p-1}\\
   &\quad  +C\sum_{j=j_1+1}^m \frac{|\mu|(\Omega_j)}{r_j^{n-1}} T_j^{2-p}
   +C\sum_{j=j_1+1}^m\omega(r_j)T_j.
\end{aligned}
\end{equation}
By \eqref{volume}, we also have
\begin{equation}\label{j1+1}
    \begin{aligned}
         \phi_{j_1+1}\leq \left(\fint_{\Omega_{r_{j_1+1}}(x_0)}|\nabla u-\mathbf{q}_{x_0,r_{j_1}}|^{\gamma_0}\right)^{1/\gamma_0}\leq C \phi_{j_1}.
    \end{aligned}
\end{equation}
Now for any $j\in\{j_0,j_0+1,\dots,j_1\}$, $r_{j}\geq \text{dist}(x_0,\partial\Omega)/2$. Choose $y_0\in \partial \Omega$ such that $d:=\text{dist}(x_0,\partial\Omega)=|y_0-x_0|$, so that $\Omega_{r_j}(x_0)\subset \Omega_{3r_j}(y_0)$ and $\Omega_{6r_j}(y_0)\subset \Omega_{8r_j}(x_0)$.
Thus by using \eqref{sum2} at $y_0\in \partial \Omega$, we have
\begin{equation}\label{sum:part2}
\begin{aligned}
&\sum_{j=j_0}^{j_1}\phi_j\leq C \sum_{j=j_0}^{j_1}\phi(y_0,3r_j)\\
&\leq  CY_{j_0}+C\sum_{j=j_0}^{j_1}\left(\frac{|\mu|(\Omega_{6r_j}(y_0))}{r_j^{n-1}}\right)^\frac{1}{p-1}\\
&\quad +C\sum_{j=j_0}^{j_1} \frac{|\mu|(\Omega_{6r_j}(y_0))}{r_j^{n-1}} Y_j^{2-p}
+C\sum_{j=j_0}^{j_1+1}\omega_1(3r_j)Y_j\\
&\leq  CZ_{j_0}+C\sum_{j=j_0}^{j_1}\left(\frac{|\mu|(\Omega_j^\prime)}{r_j^{n-1}}\right)^\frac{1}{p-1}
+C\sum_{j=j_0}^{j_1}\frac{|\mu|(\Omega_j^\prime)}{r_j^{n-1}}Z_j^{2-p}
+C\sum_{j=j_0}^{j_1+1}\omega_1(r_j)Z_j,
\end{aligned}
\end{equation}
where
$$
Y_j:=\left(\fint_{\Omega_{6r_j}(y_0)}(|\nabla u|+s)^{2-p}\,dx\right)^\frac{1}{2-p}.
$$
Recall that $\omega_1=\omega+\omega_0$. Combining \eqref{sum:part1}, \eqref{j1+1}, and \eqref{sum:part2}, we obtain \eqref{sum3}. The lemma is proved.
\end{proof}

\begin{proof}[Proof of Theorem \ref{thm:bdry}]
With \eqref{sum3} in place of \eqref{sum:int}, we can easily get the global pointwise gradient estimate \eqref{ineq:bdry1} using the same ideas as in the proof of Theorem \ref{thm:int}.
\end{proof}
\subsection{Global Lipschitz estimates and modulus of continuity estimates of the gradient}

Let $x_0\in\Bar{\Omega}$ and $0<R\leq R_1$. For any fixed ${\alpha_1}\in (0,\alpha)$, let ${\alpha_2}=({\alpha_1}+\alpha)/2$, and choose $\varepsilon=\varepsilon(n,p,\lambda,\gamma_0,\alpha,\alpha_1)\in (0,1/4)$ sufficiently small such that $\varepsilon^{\alpha_2}<1/4$ and $C\varepsilon^{\alpha-{\alpha_2}}<1$ for both constants $C$ in \eqref{ineq:prop} and \eqref{ineq:cor}. Next we define
\begin{align*}
  &g_1(x,r)=\frac{|\mu|(B_{r}(x)\cap B_{R/2}(x_0))}{r^{n-1}},\quad h_1(x,r)=g_1(x,r)^\frac{1}{p-1},\\
  &\omega^*_1(r):=\omega_1(r)[r\leq R/2]+\omega_1(R/2)[r>R/2],
\end{align*}
and
\begin{align*}
    &\Hat{g}_1(x,t)=\sum_{i=1}^\infty \varepsilon^{{\alpha_2} i}g_1(x,\varepsilon^{-i}t),\quad\Breve{g}_1(x,t)=\sum_{i=1}^\infty \varepsilon^{{\alpha_1} i}g_1(x,\varepsilon^{-i}t), \\ &\Hat{h}_1(x,t)=\sum_{i=1}^\infty \varepsilon^{{\alpha_2} i}h_1(x,\varepsilon^{-i}t),\quad\Breve{h}_1(x,t)=\sum_{i=1}^\infty \varepsilon^{{\alpha_1} i}h_1(x,\varepsilon^{-i}t),\\
    &\Hat{\omega}_1(t)=\sum_{i=1}^\infty \varepsilon^{{\alpha_2} i}\omega_1^*(\varepsilon^{-i}t),\quad
    \Breve{\omega}_1(t)=\sum_{i=1}^\infty \varepsilon^{{\alpha_1} i}\omega_1^*(\varepsilon^{-i}t),
\end{align*}
Indeed, we have
\begin{equation}\label{ineq:breve}
\begin{aligned}
     &\Breve{\omega}_1(t)=\sum_{i=1}^\infty \varepsilon^{{\alpha_1} i}\left(\omega_1(\varepsilon^{-i}t)[\epsilon^{-i}t\leq R/2]+\omega_1(R/2)[\varepsilon^{-i}t>R/2]\right):=\Tilde{\omega}_1(t),\\
&\Breve{g}_1(x,t)\leq \sum_{i=1}^\infty \varepsilon^{{\alpha_1} i}\left( g(x,\epsilon^{-i}t)[\epsilon^{-i}t\leq R/2]+g(x_0,R/2)[\varepsilon^{-i}t>R/2]\right),\\
&\Breve{h}_1(x,t)\leq \sum_{i=1}^\infty \varepsilon^{{\alpha_1} i}\left( h(x,\epsilon^{-i}t)[\epsilon^{-i}t\leq R/2]+h(x_0,R/2)[\varepsilon^{-i}t>R/2]\right),
\end{aligned}
\end{equation}
where the functions $g$ and $h$ are defined in \eqref{def:gh}.

Using the same iteration technique as in Lemma \ref{lem:iter1}, we can obtain from \eqref{ineq:cor} that
\begin{equation}\label{ineq:b}
\begin{aligned}
    \psi(x,\rho)&\leq  C\left(\frac{\rho}{r}\right)^{\alpha_2} \psi(x,r)+C\Hat{h}_1(x,2\rho)
    \\&\quad +C\Hat{g}_1(x,2\rho)\left(\|\nabla u\|_{L^\infty(\Omega_{2r}(x))}+s\right)^{2-p}\\
    &\quad+C\Hat{\omega}_1(\rho)\left(\|\nabla u\|_{L^\infty(\Omega_{2r}(x))}+s\right)
\end{aligned}
\end{equation}
holds for any $x\in \partial \Omega$, $B_{2r}(x)\subset B_{R/2}(x_0)$, and $0<\rho\leq r$.

Similarly, from \eqref{eq7.15} and the fact that $\omega\leq \omega_1$, we have
\begin{equation}\label{ineq:i}
\begin{aligned}
    \phi(x,\rho)&\leq C\left(\frac{\rho}{r}\right)^{\alpha_2} \phi(x,r)+C\Hat{h}_1(x,2\rho)\\
    &\quad +C\Hat{g}_1(x,2\rho)\left(\|\nabla u\|_{L^\infty(\Omega_{2r}(x))}+s\right)^{2-p}\\
    &\quad +C\Hat{\omega}_1(\rho)\left(\|\nabla u\|_{L^\infty(\Omega_{2r}(x))}+s\right)
\end{aligned}
\end{equation}
for any $B_{2r}(x)\subset\subset \Omega$, $B_{2r}(x)\subset B_{R/2}(x_0)$, and $0<\rho\leq r$.

By combining \eqref{ineq:b} and \eqref{ineq:i}, we will show the following estimates.
\begin{lemma}\label{lem:iter}
Let $x\in \Bar{\Omega}$ and $B_{2r}(x)\subset B_{R/2}(x_0)$. There exists a constant $C$  depending only on $\varepsilon$, $n$, $p$, $\lambda$, $\gamma_0$, and $\alpha_1$ such that, for any $0<\rho\leq r\leq R_1$, we have
\begin{itemize}
    \item[(i)]
    \begin{equation}\label{ineq:global}
\begin{aligned}
    \phi(x,\rho)\leq & C\left(\frac{\rho}{r}\right)^{\alpha_2}r^{-n/{\gamma_0}} \|\nabla u\|_{L^{\gamma_0}(\Omega_{r}(x))} +C\Breve{h}_1(x,\rho)
    \\&+C\Breve{g}_1(x,\rho)\left(\|\nabla u\|_{L^\infty(\Omega_{2r}(x))}+s\right)^{2-p}\\
    &+C\Breve{\omega}_1(\rho)\left(\|\nabla u\|_{L^\infty(\Omega_{2r}(x))}+s\right),
\end{aligned}
\end{equation}
   \item[(ii)]
   \begin{equation}\label{sum4:phi}
\begin{aligned}
    \sum_{j=0}^\infty \phi(x,\varepsilon^j\rho)
    \leq& C\left(\frac{\rho}{r}\right)^{\alpha_2} r^{-n/{\gamma_0}} \|\nabla u\|_{L^{\gamma_0}(\Omega_{r}(x))}
    +C\int_0^{\rho}\frac{\Breve{h}_1(x,t)}{t}\,dt
    \\&+C\left(\|\nabla u\|_{L^\infty(\Omega_{2r}(x))}+s\right)^{2-p}
    \int_0^{\rho}\frac{\Breve{g}_1(x,t)}{t}\,dt\\&
    +C\left(\|\nabla u\|_{L^\infty(\Omega_{2r}(x))}+s\right)
    \int_0^{\rho}\frac{\Breve{\omega}_1(t)}{t}\,dt.
\end{aligned}
\end{equation}
\end{itemize}
\end{lemma}
To prove Lemma \ref{lem:iter}, we also need the following technical lemma.

\begin{lemma}\label{tech2}
Let $x, y\in \Bar{\Omega}$ and $p\in (1,2)$.
Then for $g_1,h_1,\Hat{g}_1,\Hat{h}_1,\Breve{g}_1,\Breve{h}_1,\Hat{\omega}_1,\Breve{\omega}_1$ defined as above, we have the following:
\begin{itemize}
    \item[(i)] There exist constants $C_1,C_2>0$ depending on $\varepsilon$, $n$, $p$, $\alpha$ and $\alpha_1$ such that for any fixed $x\in\Bar{\Omega}$, and any $f\in\big\{\Hat{g}_1(x,\cdot), \Hat{h}_1(x,\cdot),\Hat{\omega}_1, \Breve{g}_1(x,\cdot), \Breve{h}_1(x,\cdot),\Breve{\omega}_1 \big\}$, we have
    $$
    C_1 f(t)\leq f(s)\leq C_2 f(t),\quad \text{whenever} \quad 0<\epsilon t\leq s\leq t.
    $$
    \item[(ii)] There exists a constant $C>0$ depending on $\varepsilon$, $n$, $p$, $\alpha$ and $\alpha_1$ such that for any $0<\epsilon r\leq \rho\leq r$ with $\Omega_{\rho}(x)\subset\Omega_{r}(y)$, and any  $F\in\{\Hat{g}_1,\Hat{h}_1,\Breve{g}_1,\Breve{h}_1\}$, we have $$F(x,\rho)\leq C F(y,r).$$
    \item[(iii)] For any $0<\rho\leq r$, there exists a constant $C>0$ depending on $\varepsilon$, $n$, $p$, $\alpha$ and $\alpha_1$ such that the following hold
    \begin{align*}
    \left(\frac{\rho}{r}\right)^{\alpha_2} \Hat{\omega}_1(r)&\leq C \Breve{\omega}_1(\rho),\\
    \left(\frac{\rho}{r}\right)^{\alpha_2} \Hat{g}_1(x,r)&\leq C \Breve{g}_1(x,\rho),\\
    \left(\frac{\rho}{r}\right)^{\alpha_2} \Hat{h}_1(x,r)&\leq C \Breve{h}_1(x,\rho).
    \end{align*}
\end{itemize}
\end{lemma}
\begin{proof}
We will only show the proof for $g$ since the other cases are similar. Noting that $$
g_1(x,s)\leq \epsilon^{1-n} g_1(x,t),\quad g_1(x,\epsilon t)\leq \epsilon^{1-n} g_1(x,s)
$$
whenever $\epsilon t\leq s\leq t$ and
$\Hat{g}_1(x,t)\leq \varepsilon^{-{\alpha_2}} \Hat{g}_1(x,\epsilon t)$, assertion $(i)$ follows. Assertion $(ii)$ follows similarly by observing the fact that $\Omega_{\varepsilon^{-i}\rho}(x)\subset\Omega_{\varepsilon^{-i}r}(y)$ whenever $i\geq 0$ since $\Omega_{\rho}(x)\subset\Omega_{r}(y)$.
It remains to prove assertion $(iii)$. Since $0<\rho\leq r$, there exists an integer $j\geq 0$ such that $\varepsilon^{-j}\rho\leq r< \varepsilon^{-j-1}\rho$. Therefore, by part $(i)$,
\begin{align*}
    &\left(\frac{\rho}{r}\right)^{\alpha_2}\Hat{g}_1(x,r) \leq C \varepsilon^{{\alpha_2} j}\Hat{g}_1(x,\varepsilon^{-j}\rho)\\
    &\leq C\sum_{j=0}^\infty \sum_{i=1}^\infty \varepsilon^{{\alpha_2} (i+j)} g_1(x,\varepsilon^{-i-j}\rho)\\
    &=C \sum_{k=1}^\infty k\varepsilon^{{\alpha_2} k}g_1(x,\varepsilon^{-k}\rho)\leq C \Breve{g}_1(x,\rho),
\end{align*}
where we used the fact that $k\varepsilon^{{\alpha_2} k} \leq C \varepsilon^{{\alpha_1} k}$ in the last inequality, since ${\alpha_1}<{\alpha_2}$.
\end{proof}

Now we are ready to prove Lemma \ref{lem:iter}.

\begin{proof}[Proof of Lemma \ref{lem:iter}]
Without loss of generality, we may assume $x=0$. Note that if $r/16\leq\rho\leq r$, then \eqref{ineq:global} follows from the definition of $\phi$. Hence we only need to consider the case when $0<\rho<r/16$. We consider the following three cases:
$$r/4\leq \text{dist}(0,\partial \Omega),\quad \text{dist}(0, \partial \Omega)\leq 4\rho,\quad 4\rho< \text{dist}(0,\partial \Omega)<r/4.$$

{\em Case 1: $r/4\leq \text{dist}(0,\partial \Omega)$.} Set $r_1=r/16$. Since $B_{4r_1}\subset \Omega$, from \eqref{ineq:i} we have
\begin{align*}
\phi(0,\rho)&\leq C\left(\frac{\rho}{r_1}\right)^{\alpha_2} \phi(0,r_1)+C\Hat{h}_1(0,2\rho)\\
    &\quad+C\Hat{g}_1(0,2\rho)\left(\|\nabla u\|_{L^\infty(\Omega_{2r_1})}+s\right)^{2-p}+C\Hat{\omega}_1(\rho)\left(\|\nabla u\|_{L^\infty(\Omega_{2r_1})}+s\right).
\end{align*}
Thus we can easily get \eqref{ineq:global} from Lemma \ref{tech2} and the fact that
\begin{align*}
    \phi(0,r_1)\leq Cr^{-n/{\gamma_0}} \|\nabla u\|_{L^{\gamma_0}(\Omega_{r})}.
\end{align*}

{\em Case 2: $\text{dist}(0, \partial \Omega)\leq 4\rho$.} Choose $y_0\in \partial \Omega$ such that $\text{dist}(0,\partial \Omega)=|y_0|$. Then $B_{10 \rho}(y_0)\subset B_{14\rho}\subset B_{r}$ and from \eqref{ineq:b}, we have
\begin{align*}
    &\phi(0,\rho)\leq  C \psi(y_0,5\rho)\\
    &\leq C\left(\frac{\rho}{r}\right)^{\alpha_2} \psi(y_0,r/2)+C\Hat{h}_1(y_0,10\rho)
    \\&\quad +C\Hat{g}_1(y_0,10\rho)\left(\|\nabla u\|_{L^\infty(\Omega_{r}(y_0))}+s\right)^{2-p}+C\Hat{\omega}_1(5\rho)\left(\|\nabla u\|_{L^\infty(\Omega_{r}(y_0))}+s\right).
\end{align*}
Thus from the fact that $\Omega_r(y_0)\subset \Omega_{2r}$, $\Omega_{r/2}(y_0)\subset \Omega_r$, and $\Omega_{10 \rho}(y_0)\subset \Omega_{14\rho}$, we get \eqref{ineq:global} by using Lemma \ref{tech2}.

{\em Case 3: $4 \rho< \text{dist}(0,\partial \Omega)<r/4$.} Set $r_1=\text{dist}(0,\partial \Omega)/4>\rho$. Using \eqref{ineq:i}, we obtain
\begin{align*}
\phi(0,\rho)&\leq C\left(\frac{\rho}{r_1}\right)^{\alpha_2} \phi(0,r_1)+C\Hat{h}_1(0,2\rho)\\
    &\quad+C\Hat{g}_1(0,2\rho)\left(\|\nabla u\|_{L^\infty(\Omega_{2r_1})}+s\right)^{2-p}+C\Hat{\omega}_1(\rho)\left(\|\nabla u\|_{L^\infty(\Omega_{2r_1})}+s\right).
\end{align*}
On the other hand, choose $y_0\in \partial \Omega$ such that
$\text{dist}(0,\partial \Omega)=|y_0|$. Therefore, $B_{r_1}\subset B_{5r_1}(y_0), $ $B_{r}(y_0)\subset B_{2r}$ and from \eqref{ineq:b} we have
\begin{align*}
    &\phi(0,r_1)\leq  C \psi(y_0,5r_1)\\
    &\leq  C\left(\frac{r_1}{r}\right)^{\alpha_2} \psi(y_0,r/2)+C\Hat{h}_1(y_0,10r_1)
    \\&\quad +C\Hat{g}_1(y_0,10r_1)\left(\|\nabla u\|_{L^\infty(\Omega_{r}(y_0))}+s\right)^{2-p}+C\Hat{\omega}_1(5r_1)\left(\|\nabla u\|_{L^\infty(\Omega_{r}(y_0))}+s\right).
\end{align*}
Noting that $\Omega_r(y_0)\subset \Omega_{2r}$, $\Omega_{r/2}(y_0)\subset \Omega_r$, and $\Omega_{10 r_1}(y_0)\subset \Omega_{14r_1}$, we get \eqref{ineq:global} by combining the last two estimates and applying Lemma \ref{tech2}.

Finally, replacing $\rho$ with $\varepsilon^j \rho$ and summing in $j$, we get \eqref{sum4:phi} by using Lemma \ref{tech2} and the comparison principle of Riemann integrals. The lemma is proved.
\end{proof}

\begin{remark}\label{local}
We emphasize that Lemma \ref{lem:iter} has a local nature. Indeed, it can be seen from the proof that we only need Dirichlet boundary condition $u=0$ on $\partial\Omega\cap B_{R/2}(x_0)$ and $C^{1,\text{Dini}}$ regularity of $\partial\Omega\cap B_{R/2}(x_0)$ for these estimates to hold. Therefore, our Lipschitz estimates and modulus of continuity estimates, which will be deduced from
Lemma \ref{lem:iter}, also have a local nature.
\end{remark}

Recall the definition of $\mathbf{q}_{x,r}$ from \eqref{def:qxr} and keep \eqref{volume} in mind. By following almost the same proof of \eqref{ineq:diff}, we obtain that for any Lebesgue point $x\in\Omega$ of the vector-valued function $\nabla u$ and $\rho\in (0,R_1]$,
\begin{equation}\label{ineq:diff2}
     |\nabla u(x)-\mathbf{q}_{x,\rho}|\leq C\sum_{j=0}^\infty \phi(x,\varepsilon^j\rho),
\end{equation}
where $C$ is a constant depending only on $n$ and $\gamma_0$.

\begin{proof}[Proof of Theorem \ref{thm:bdry:lip}]

We will prove a boundary Lipschitz estimate
\begin{equation}\label{ineq:bdry:lip2}
\|\nabla u\|_{L^\infty(\Omega_{R/16}(x_0))}
         \leq C\big\|\mathbf{I}_1^R(|\mu|)\big\|^\frac{1}{p-1}_{L^\infty(\Omega_{R}(x_0))}+ CR^{-\frac{n}{2-p}} \||\nabla u|+s\|_{L^{2-p}(\Omega_{R}(x_0))}
  \end{equation}
for any $x_0\in \partial\Omega$ and $R\leq R_1$, assuming that
\begin{equation}\label{finite Riesz bdry}
    \big\|\mathbf{I}_1^R(|\mu|)\big\|_{L^\infty(B_{R}(x_0))}<\infty.
\end{equation}
Then \eqref{ineq:bdry:lip} follows by a standard covering argument using \eqref{ineq:int:lip} and \eqref{ineq:bdry:lip2}.
The proof of \eqref{ineq:bdry:lip2} is similar to that of Theorem \ref{thm:int:lip} so we will only focus on the differences.

{\em Step 1: The case when $u\in C^1(\overline{\Omega_{R/2}(x_0)})$.}

With \eqref{sum4:phi} in place of \eqref{sum1:phi}, using the same iteration technique as in the proof of Theorem \ref{thm:int:lip}, we get the following estimate:
\begin{equation}\label{infty-norm-bdry}
    \begin{aligned}
         &\|\nabla u\|_{L^\infty(\Omega_{R/4}(x_0))}+s\leq C R^{-n/{\gamma_0}} \|\nabla u\|_{L^{\gamma_0}(\Omega_{R/2}(x_0))}
   \\ &\qquad+C   \sup_{x\in \Omega_{R/2}(x_0)}\int_0^{\frac{R}{2}}\frac{\Breve{h}_1(x,t)}{t}\,dt
    +C \sup_{x\in \Omega_{R/2}(x_0)} \left(\int_0^{\frac{R}{2}}\frac{\Breve{g}_1(x,t)}{t}\,dt\right)^{\frac{1}{p-1}}+Cs.
    \end{aligned}
\end{equation}
Using \eqref{ineq:breve} and direct computations, we have
\begin{align*}
     &\int_0^{{R/2}}\frac{\Breve{g}_1(x,t)}{t}\,dt\leq C\,\mathbf{I}_1^{{R}}(|\mu|)(x)+C\,\frac{|\mu|(B_{R/2}(x_0))}{R^{n-1}},\\
     &\int_0^{R/2}\frac{\Tilde{h}_1(x,t)}{t}\,dt\leq C \left(\mathbf{I}_1^{{R}}(|\mu|)(x)\right)^\frac{1}{p-1}+C\left(\frac{|\mu|(B_{R/2}(x_0))}{R^{n-1}}\right)^\frac{1}{p-1}.
\end{align*}
Therefore, from \eqref{infty-norm-bdry} and the fact that $\gamma_0\leq 2-p$ (cf. Lemma \ref{lem:u-w2}), we obtain
\begin{equation}\label{infty-norm-bdry2}
        \|\nabla u\|_{L^\infty(\Omega_{R/4}(x_0))}
         \leq C\big\|\mathbf{I}_1^R(|\mu|)\big\|^\frac{1}{p-1}_{L^\infty(\Omega_{R}(x_0))}+ CR^{-\frac{n}{2-p}} \||\nabla u|+s\|_{L^{2-p}(\Omega_{R}(x_0))}.
\end{equation}
{\em Step 2: The general case.}

We use an approximation argument with the aid of the regularized distance introduced by Lieberman \cite{Lieb85}. Here we refer to a modified version in \cite{dong2020on}.
Let $d(\cdot)$ be the regularized distance defined in \cite[Lemma 5.1]{dong2020on} ($\psi(\cdot)$ in that paper) and  $\Omega^{k}=\{x\in \Omega: d(x)>1/k\}$. Then from \cite[Lemma 5.1]{dong2020on}, we know that $\Omega^k$ has a smooth boundary and the $C^{1,\text{Dini}}$-properties of $\partial\Omega_k$ are the same as those of $\partial\Omega$ up to some constant independent of $k$.
We take a sequence of standard mollifiers $\{\varphi_k\}$ and mollify $\mu$ and $a$ by setting
$$\mu_k(x)=(\mu*\varphi_k)(x),\quad x\in \Omega; \quad a^k(x)=(a*\varphi_k)(x), \quad x\in \Omega^k.$$
We know that $\mu\in W^{-1,p'}(\Omega_R(x_0))$ and therefore
\begin{equation*}
 \|\mu_k-\mu\|_{W^{-1,p'}(\Omega_R(x_0))}\to 0.
\end{equation*}
Recalling the fact that we have a $C^{1,\text{Dini}}$ coordinate in $\Omega_R(x_0)$ since $R\leq R_1$, we can take a sequence of cut-off functions $\zeta_k\in C^\infty(\R^n)$ satisfying
$\zeta_k=1$ in $\Omega^{k/4}\cap B_R(x_0)$, $\zeta_k=0$ in $(\Omega\backslash\Omega^{k/2})\cap B_R(x_0)$, and $\|\nabla \zeta_k\|_{L^{\infty}}\leq 16k$.

Next we let $u_k\in u\zeta_k+W^{1,p}_0(\Omega^k\cap B_R(x_0))$ be the unique solution to
\begin{equation}\label{eq:uk2}
    \left\{
\begin{aligned}
     -\div\left(a^k(x)(|\nabla u_k|^2+s^2)^\frac{p-2}{2}\nabla u_k\right) =&  \mu_k\quad \text{in} \quad \Omega^k\cap B_R(x_0), &\\
     u_k =&  u\zeta_k \quad \text{on}  \quad \partial (\Omega^k\cap B_R(x_0)).  &\\
\end{aligned}
\right.
\end{equation}
Since $u_k=u\zeta_k=0$ on $B_R(x_0)\,\cap\,\partial \Omega^k$, we can always assume $u_k\in u\zeta_k+ W^{1,p}_0(\Omega_R(x_0))$ by taking the zero extension of $u_k$ in $(\Omega\backslash \Omega^k)\cap B_R(x_0)$.
Since $u=0$ on $B_R(x_0)\,\cap\,\partial \Omega$, by Hardy's inequality, we have
\begin{align*}
    &\|u\nabla\zeta_k\|_{L^p\big(\Omega_R(x_0)\big)}\leq 16 \|ku\|_{L^p\big((\Omega\backslash\Omega^{k/4})\cap B_R(x_0)\big)}\\
    &\leq C\|{u(x)}/{d(x)}\|_{L^p\big((\Omega\backslash\Omega^{k/4})\cap B_R(x_0)\big)}
    \leq C\|\nabla u\|_{L^p\big((\Omega\backslash\Omega^{k/4})\cap B_R(x_0)\big)}\to 0
\end{align*}
as $k\to \infty$.
Therefore, we know that
\begin{equation}\label{cutconv}
    \|u-u\zeta_k\|_{W^{1,p}(\Omega_R(x_0))}\to 0.
\end{equation}
Thus by choosing $u_k-u\zeta_k$ as a test function in \eqref{eq:uk2}, following the proof of \eqref{bound1}, and using \eqref{cutconv}, we can show that
$\|\nabla u_k\|_{L^{p}(\Omega_R(x_0))}$ is uniformly bounded in $k$. Also, choosing $(u_k-u\zeta_k)1_{\Omega_R(x_0)}$ as a test function in \eqref{eq:p} and \eqref{eq:uk2}, similarly we obtain
$$
\int_{\Omega_R(x_0)}|V(\nabla u_k)-V(\nabla u)|^2\,dx\to 0\quad \text{as}\,k\to\infty,
$$
which again implies
$$\nabla u_k\to \nabla u \quad \text{strongly in } L^{p}(\Omega_R(x_0)).$$
By the classical boundary regularity theory (see, for instance, \cite{MR969499}), we have $u_k\in C^1(\overline{\Omega^k\cap B_{R/2}(x_0)})$.
Note that for sufficiently large $k$, there exists $x_k\in B_R(x_0)\,\cap\,\partial \Omega^k$, such that $|x_k-x_0|\leq R/16$. Thus $B_{R/16}(x_0)\subset B_{R/8}(x_k)$, $B_{R/4}(x_k)\subset B_{R/2}(x_0)$ and $B_{R/2}(x_k)\subset B_{3R/4}(x_0)$.
Therefore, using \eqref{infty-norm-bdry2} in Step 1 and Remark \ref{local}, we get
\begin{align*}
   &\|\nabla u_k\|_{L^\infty\big(\Omega^k\cap B_{R/16}(x_0)\big)}\leq \|\nabla u_k\|_{L^\infty\big(\Omega^k\cap B_{R/8}(x_k)\big)}\\
         &\leq C\big\|\mathbf{I}_1^{R/2}(|\mu_k|)\big\|^\frac{1}{p-1}_{L^\infty\big(\Omega^k\cap B_{R/2}(x_k)\big)}+ CR^{-\frac{n}{2-p}} \||\nabla u_k|+s\|_{L^{2-p}\big(\Omega^k\cap B_{R/2}(x_k)\big)}\\
  &\leq C\big\|\mathbf{I}_1^{R/2}(|\mu_k|)\big\|^\frac{1}{p-1}_{L^\infty\big(\Omega^k\cap B_{3R/4}(x_0)\big)}+ CR^{-\frac{n}{2-p}} \||\nabla u_k|+s\|_{L^{2-p}\big(\Omega^k\cap B_{3R/4}(x_0)\big)}.
\end{align*}
By extracting a subsequence and taking the limit as $k\to \infty$, we obtain \eqref{ineq:bdry:lip2}.
\end{proof}

\begin{proof}[Proof of Corollary \ref{thm:glo:lip}]
By testing \eqref{eq:p} with $u$, following the proof of \eqref{bound1}, we obtain
$$
\|\nabla u\|_{L^p(\Omega)}\leq C\|\mu\|^\frac{1}{p-1}_{W^{-1,p'}(\R^n)}+Cs,
$$
where $p'=p/(p-1)$.
From \cite[Theorem 1]{hedberg1983thin}, we also have
$$
\|\mu\|_{W^{-1,p'}(\R^n)}\leq C\Big(\int_{\R^n} \mathbf{W}_{1,p}^1(|\mu|)d|\mu|\Big)^{\frac{p-1}{p}}\leq  C\big\|\mathbf{I}^{1}_1(|\mu|)\big\|_{L^{\infty}(\Omega)}.
$$
Therefore, Corollary \ref{thm:glo:lip} follows by combining \eqref{ineq:bdry:lip}, H\"older's inequality, and the last two inequalities.
\end{proof}

Now we turn to global modulus of continuity estimates of the gradient.
Recall that we fixed an $\varepsilon\in (0,1/4)$ sufficiently small such that
$$
C\varepsilon^{\alpha-{\alpha_2}}<1\quad \text{and}\quad
\varepsilon^{\alpha_2}<1/4
$$
for both constants $C$ in \eqref{ineq:prop} and \eqref{ineq:cor}, where $\alpha\in(0,1)$ is the same constant as in Theorem \ref{thm:osc}, ${\alpha_1}\in (0,\alpha)$, and $\alpha_2=(\alpha_1+\alpha)/2$. We also took $R\in(0, R_1]$ and defined
\begin{equation}\label{def:tilde2}
\begin{aligned}
 &\Tilde{\omega}_1(t)=\sum_{i=1}^\infty \varepsilon^{{\alpha_1} i}\left(\omega_1(\varepsilon^{-i}t)[\epsilon^{-i}t\leq R/2]+\omega_1(R/2)[\varepsilon^{-i}t>R/2]\right),\\
 &\Tilde{\mathbf{I}}_1^\rho(|\mu|)(x)=\sum_{i=1}^\infty \varepsilon^{{\alpha_1} i}\left(\mathbf{I}_1^{\varepsilon^{-i}\rho}(|\mu|)(x)[\varepsilon^{-i}\rho\leq R/2]+\mathbf{I}_1^{R/2}(|\mu|)(x)[\varepsilon^{-i}\rho>R/2]\right),\\
&\Tilde{\mathbf{W}}_{1/p,p}^\rho(|\mu|)(x)\\
&=\sum_{i=1}^\infty \varepsilon^{{\alpha_1} i}\left(\mathbf{W}_{1/p,p}^{\varepsilon^{-i}\rho}(|\mu|)(x)[\varepsilon^{-i}\rho\leq R/2]+\mathbf{W}_{1/p,p}^{R/2}(|\mu|)(x)[\varepsilon^{-i}\rho>R/2]\right),
\end{aligned}
\end{equation}
where we used the Iverson bracket notation, i.e., $[P] = 1$ if $P$ is true and $[P] = 0$ otherwise, and $\mathbf{I}_1$ and $\mathbf{W}_{1/p,p}$ are the Riesz and Wolff potentials defined in \eqref{def:riesz} and \eqref{def:wolff}, respectively.

Our global modulus of continuity estimate of the gradient is stated as follows.
\begin{theorem}\label{thm:bdry2}
Assume the conditions of Theorem \ref{thm:bdry:lip} and ${\alpha_1}\in (0,\alpha)$, where $\alpha$ is the constant in Theorem \ref{thm:osc}. Then there exist constants $R_1=R_1(R_0,\omega_0)\in(0,R_0)$
and $C=C(n,p,\lambda,\alpha_1, \omega,R_0,\omega_0)$, such that for any $x_0\in \Bar{\Omega}$, $R\in (0,R_1]$, and $x,y\in \Omega_{R/4}(x_0)$ being Lebesgue points of the vector-valued function $\nabla u$, it holds that
\begin{equation}\label{cont:glob}
\begin{aligned}
    &|\nabla u(x)-\nabla u(y)|\\
    &\leq C\,\mathbf{M}_1\Big[\left(\frac{\rho}{R}\right)^{\alpha_1}+\int_0^{\rho}\frac{\Tilde{\omega}_1(t)}{t}\,dt\Big]
    +C\,\big\|\Tilde{\mathbf{W}}_{1/p,p}^{\rho}(|\mu|)\big\|_{L^\infty(\Omega_{R/4}(x_0))}\\
    &\quad +C\,\mathbf{M}_1^{2-p}
    \big\|\Tilde{\mathbf{I}}_1^{\rho}(|\mu|)\big\|_{L^\infty(\Omega_{R/4}(x_0))}
\end{aligned}
\end{equation}
where $\rho=|x-y|$, $\omega_1=\omega+\omega_0$, $\Tilde{\omega}_1$, $\Tilde{\mathbf{W}}_{1/p,p}$, and $\Tilde{\mathbf{I}}_1$ are defined in \eqref{def:tilde2}, and
$$
\mathbf{M}_1:=R^{-\frac{n}{2-p}}\||\nabla u|+s\|_{L^{{2-p}}(\Omega_{R}(x_0))}+\big\|\mathbf{I}_1^{R}(|\mu|)\big\|^\frac{1}{p-1}_{L^\infty(\Omega_{R}(x_0))}.
$$
\end{theorem}

\begin{proof}
For any $x,y\in \Omega_{R/4}(x_0)$ being Lebesgue points of $\nabla u$ and $\rho>0$,
\begin{align*}
    &|\nabla u(x)-\nabla u(y)|^{\gamma_0}\\
    &\leq |\nabla u(x)-\mathbf{q}_{x,\rho}|^{\gamma_0}+|\nabla u(y)-\mathbf{q}_{y,2\rho}|^{\gamma_0}+|\mathbf{q}_{x,\rho}-\mathbf{q}_{y,2\rho}|^{\gamma_0}\\
    &\leq |\nabla u(x)-\mathbf{q}_{x,\rho}|^{\gamma_0}+ |\nabla u(y)-\mathbf{q}_{y,2\rho}|^{\gamma_0}+|\nabla u(z)-\mathbf{q}_{x,\rho}|^{\gamma_0}+|\nabla u(z)-\mathbf{q}_{y,2\rho}|^{\gamma_0}.
\end{align*}
We set $\rho=|x-y|$, take the average over $z\in \Omega_\rho(x)$, and then take the $\gamma_0$-th root to get
\begin{align*}
    &|\nabla u(x)-\nabla u(y)|\\
    &\leq C\; |\nabla u(x)-\mathbf{q}_{x,\rho}|^{\gamma_0}+C\;|\nabla u(y)-\mathbf{q}_{y,2\rho}|^{\gamma_0}+C\phi(x,\rho)+C\phi(y,2\rho)\\
    &\leq C \sum_{j=0}^\infty \phi(x,\varepsilon^j\rho)+C \sum_{j=0}^\infty \phi(y,2\varepsilon^j \rho)+ C\phi(x,\rho)+C\phi(y,2\rho)\\
    &\leq  C \sup_{y_0\in \Omega_{R/4}(x_0)}\sum_{j=0}^\infty \phi(y_0,2\varepsilon^j\rho),
\end{align*}
where we used the fact that $\Omega_\rho(x)\subset \Omega_{2\rho}(y)$ in the first inequality and (\ref{ineq:diff2}) in the second inequality.

If $\rho< R/16$, by using (\ref{sum4:phi}) with $R/8$ in place of $r$ and the fact that
\begin{align*}
    \Omega_{R/4}(y_0)\subset \Omega_{R/2}(x_0)\quad \forall y_0\,\in \Omega_{R/4}(x_0),
\end{align*}
we obtain
\begin{equation}\label{eq:diff2}
\begin{aligned}
     &|\nabla u(x)-\nabla u(y)|\\
     &\leq  C \left(\frac{\rho}{R}\right)^{\alpha_2} \|\nabla u\|_{L^\infty(\Omega_{R/2}(x_0))}+C \sup_{y_0\in \Omega_{R/4}(x_0)}\int_0^{\rho}\frac{\Breve{h}_1(y_0,t)}{t}\,dt
    \\&\quad+C\left(\|\nabla u\|_{L^\infty(\Omega_{R/2}(x_0))}+s\right)^{2-p}\sup_{y_0\in \Omega_{R/4}(x_0)}\int_0^{\rho}\frac{\Breve{g}_1(y_0,t)}{t}\,dt
    \\&\quad
    +C\left(\|\nabla u\|_{L^\infty(\Omega_{R/2}(x_0))}+s\right)\int_0^{\rho}\frac{\Breve{\omega}_1(t)}{t}\,dt.
\end{aligned}
\end{equation}
Clearly, (\ref{eq:diff2}) still holds when $\rho\geq R/16$.
Using \eqref{ineq:breve} and similar calculations as in the proof of Theorem \ref{thm4.1}, for any $y_0\in \Omega_{R/4}(x_0)$ and $\rho\in (0,R/2)$, we have
\begin{align}
                    \label{g2}
    &\int_0^{\rho}\frac{\Breve{g}_1(y_0,t)}{t}\,dt\leq \Tilde{\mathbf{I}}_1^{\rho}(|\mu|)(y_0)+C\left(\frac{\rho}{R}\right)^{\alpha_1}\frac{|\mu|(B_{R/2}(x_0))}{R^{n-1}},\\
\label{h2}
    &\int_0^\rho\frac{\Breve{h}_1(y_0,t)}{t}\,dt\leq \Tilde{\mathbf{W}}_{1/p,p}^{\rho}(|\mu|)(y_0)+C\left(\frac{\rho}{R}\right)^{\alpha_1}\left(\frac{|\mu|(B_{R/2}(x_0))}{R^{n-1}}\right)^\frac{1}{p-1}.
\end{align}
Using \eqref{ineq:bdry:lip}, \eqref{g2}, and \eqref{h2}, \eqref{eq:diff2} implies \eqref{cont:glob}. The theorem is proved.
\end{proof}

\section*{Acknowledgements}
The authors would like to thank Quoc-Hung Nguyen for helpful comments and pointing out a gap in an earlier version of the manuscript.

\bibliographystyle{plain}

\end{document}